\RequirePackage[table]{xcolor}
\documentclass[a4paper,12pt]{amsart}
\usepackage{amsfonts}
\usepackage{amsmath}
\numberwithin{equation}{section}
\usepackage{ifthen}
\usepackage{amsrefs}
\usepackage{amsthm}
\usepackage[all,cmtip]{xy}
\usepackage{amssymb}
\usepackage{hyperref}

\usepackage{amsmath, bm}
\usepackage{tikz-cd}
\usepackage{pifont}
\usepackage{mathtools}

\usepackage{hyperref}
\usepackage{url}
\usepackage[shortlabels]{enumitem}
\usepackage[paper=a4paper,left=20mm,right=20mm,top=25mm,bottom=30mm]{geometry}
\setlist[enumerate]{topsep=0em, itemsep= -0em, parsep = 0 em, label=$(\alph*)$}

\nocite{*}

\newcommand{\cF}{\mathcal{F}}

\newcommand{\cK}{\mathcal{K}}

\newcommand{\cT}{\mathcal{T}}

\newcommand{\cX}{\mathcal{X}}
\newcommand{\cY}{\mathcal{Y}}

\DeclareMathOperator{\Hom}{Hom}

\DeclareMathOperator{\Ext}{Ext}

\DeclareMathOperator{\End}{End}

\DeclareMathOperator{\Tr}{Tr}

\DeclareMathOperator{\Tor}{Tor}

\newtheorem{proposition}{Proposition}
\newtheorem{Theorem}[proposition]{Theorem}

\newtheorem{TheoremS}{Theorem}[section]
\newtheorem{CorollaryS}[TheoremS]{Corollary}
\newtheorem{LemmaS}[TheoremS]{Lemma}
\newtheorem{PropositionS}[TheoremS]{Proposition}

\newenvironment{example}[1][Example.]{\begin{trivlist}
\item[\hskip \labelsep {\bfseries #1}]}{\end{trivlist}}

\newenvironment{Remark}[1][Remark.]{\begin{trivlist}
\item[\hskip \labelsep {\bfseries #1}]}{\end{trivlist}}

\newenvironment{Definition}[1][Definition.]{\begin{trivlist}
\item[\hskip \labelsep {\bfseries #1}]}{\end{trivlist}}

\BibSpec{collection.article}{%
    +{}  {\PrintAuthors}                {author}
    +{,} { \textit}                     {title}
    +{,} { }                            {series}
    +{,} { }                            {volume}
    +{.} { }                            {part}
    +{:} { \textit}                     {subtitle}
    +{,} { \PrintContributions}         {contribution}
    +{,} { \PrintConference}            {conference}
    +{}  {\PrintBook}                   {book}
    +{,} { }                            {booktitle}
    +{,} { }                            {publisher}
    +{,} { \textit}                     {address}
    +{,} { \PrintDateB}                 {date}
    +{,} { pp.~}                        {pages}
    +{,} { }                            {status}
    +{,} { \PrintDOI}                   {doi}
    +{,} { available at \eprint}        {eprint}
    +{}  { \parenthesize}               {language}
    +{}  { \PrintTranslation}           {translation}
    +{;} { \PrintReprint}               {reprint}
    +{.} { }                            {note}
    +{.} {}                             {transition}
    +{}  {\SentenceSpace \PrintReviews} {review}
}

\begin{document}

\rmfamily


\thispagestyle{empty}

\title{Tilting quivers for BB-tilted algebras} 

\date{}
\author{Hongwei Peng}

\newcommand{\Addresses}{{
  \bigskip
  \footnotesize

 \textsc{Christian-Albrechts-Universität zu Kiel, Heinrich-Hecht-Platz 6, 24118 Kiel, Germany
}\par\nopagebreak
  \textit{E-mail address}: \texttt{peng@math.uni-kiel.de}

  }}

\maketitle

\begin{abstract}
Let $\Lambda$ be a hereditary algebra, $B_0=\End_\Lambda(T_0)$ be a tilted algebra. We will construct tilting $B_0$-modules from tilting $\Lambda$-modules and use this result to show how tilting quivers of BB-tilted algebras can be obtained from those of $\Lambda$.
\end{abstract}

Tilting theory, initiated in \cite{APR}, \cite{BB}, further developed in \cite{KB}, \cite{HR} and then generalized in \cite{M}, is an interesting and important topic in representation theory (for more information about tilting theory, we refer to \cite{HHK}). The construction of tilting modules has been a central topic with emphasis on finding complements to almost complete tilting modules (see \cite{HaUn4}, \cite{CHU}). In this article, we will give a way of constructing tilting modules from a given one. Our main result reads as follows:

\begin{Theorem}\label{The 1}
Let $B_0=\End_\Lambda(T_0)$ be a tilted algebra of type $\Lambda$. If $T=Y\oplus X$ is a tilting $\Lambda$-module with $Y\in\mathcal{T}(T_0)$, $X\in\mathcal{F}(T_0)$ such that no indecomposable summand of $Y$ is generated by $X$, then $\Hom_\Lambda(T_0,Y/Tr_{Y}X)\oplus \Ext^{1}_\Lambda(T_0,X)$ is a tilting $B_0$-module.
\end{Theorem}

\noindent Here $\Lambda$ is a hereditary algebra, $(\mathcal{T}(T_0),\mathcal{F}(T_0))$ is the torsion pair in mod $\Lambda$ associated to $T_0$, and $Tr_{Y}X$ denotes the trace of $X$ in $Y$.

Let $A$ be a finite dimensional $k$ algebra over an algebraically closed field $k$, $\mathcal{T}_A$ be the set of equivalence classes of tilting $A$-modules (tilting modules $T_1$ and $T_2$ are supposed to be in the same class, provided $\mathsf{add}(T_1)=\mathsf{add}(T_2)$). Here $\mathsf{add}(T)\subset \text{mod }A$ denotes the additive closure of $T$. The tilting quiver $\overrightarrow{\mathcal{K}_A}$ was introduced in \cite{HaUn2}, \cite{RiS}: the vertices are elements of $\mathcal{T}_A$, and for two basic tilting $A$-modules, there exists an arrow from $T_1$ to $T_2$ if and only if $T_1=M\oplus X$, $T_2=M\oplus Y$ with $X,Y$ being indecomposable summands belonging to a non-split exact sequence $0\to X\to \overline{M}\to Y\to 0$ with $\overline{M}\in \mathsf{add} M$. Given $T_1,T_2\in\mathcal{T}_A$, a partial order $\leq$ on $\mathcal{T}_A$ is defined via: $T_1\leq T_2$ if and only if $\mathsf{T_1}^{\perp}\subset \mathsf{T_2}^{\perp}$. Here $\mathsf{T_i}^{\perp}$ are full subcategories of mod $A$ (see below). It was proved in \cite{HaUn2}, that the Hasse quiver of the partially ordered set $(\mathcal{T}_A,\leq)$ coincides with the tilting quiver $\overrightarrow{\mathcal{K}_A}$. $\overrightarrow{\mathcal{K}_A}$ turned out to be interesting and useful, because it provides  information about $A$. Under some assumptions on $A$, one can even reconstruct the algebra A from the partially ordered set $(\mathcal{T}_A,\leq)$. The reader is referred to  \cite{HaUn2}, \cite{HaUn3}, \cite {HaUn1} for further details.

Let $T_0$ be a classic tilting $A$-module, $B_0=\End_A(T_0)$. According to the Tilting Theorem (see Theorem \ref{Th0} below), there exist equivalences $\mathcal{T}(T_0)\cong \mathcal{Y}(T_0)$, and $\mathcal{F}(T_0)\cong \mathcal{X}(T_0)$. Here $\mathcal{T}(T_0)$ and $\mathcal{F}(T_0)$, $\mathcal{X}(T_0)$ and $\mathcal{Y}(T_0)$ are full subcategories given by torsion pairs in mod $A$ and mod $B_0$, respectively. Let $\mathcal{T}_{\mathcal{T}}\subset\mathcal{T}_A$ be the subset of $\mathcal{T}_A$ consisting of those equivalence classes whose representatives belong to $\cT(T_0)$; $\mathcal{T}_{\mathcal{F}}\subset\mathcal{T}_A$ be the subset of $\mathcal{T}_A$ consisting of those equivalence classes whose representatives belong to $\mathcal{F}(T_0)$;  
$\mathcal{T}_{\mathcal{T}\cup\mathcal{F}}$ be the subset of $\mathcal{T}_A$ consisting of  equivalence classes whose representatives have indecomposable direct summands only in $\mathcal{T}(T_0)\cup \mathcal{F}(T_0)$  
and let $\mathcal{T}_{\mathcal{T},\mathcal{F}}\subset\mathcal{T}_A$ be the subset of $\mathcal{T}_A$ consisting of elements in $\mathcal{T}_{\mathcal{T}\cup\mathcal{F}}$ but not in the union of $\mathcal{T}_{\mathcal{T}}$ and $\mathcal{T}_{\mathcal{F}}$. Similarly, we can define $\mathcal{T}_{\mathcal{Y}},\mathcal{T}_{\mathcal{X}},\mathcal{T}_{\mathcal{X}\cup\mathcal{Y}}\subset\mathcal{T}_{B_0} $ and $\mathcal{T}_{\mathcal{X},\mathcal{Y}}\subset\mathcal{T}_{B_0}$. It is quite reasonable to consider the relationship between $\overrightarrow{\mathcal{T}_{\mathcal{T}}}$ and $\overrightarrow{\mathcal{T}_{\mathcal{Y}}}$, $\overrightarrow{\mathcal{T}_{\mathcal{F}}}$ and $\overrightarrow{\mathcal{T}_{\mathcal{X}}}$, 
$\overrightarrow{\mathcal{T}_{\mathcal{T}\cup\mathcal{F}}}$ and $\overrightarrow{\mathcal{T}_{\mathcal{X}\cup\mathcal{Y}}}$, $\overrightarrow{\mathcal{T}_{\mathcal{T},\mathcal{F}}}$ and $\overrightarrow{\mathcal{T}_{\mathcal{X},\mathcal{Y}}}$, and then the relationship between $\overrightarrow{\mathcal{K}_A}$ and $\overrightarrow{\mathcal{K}_{B_0}}$. For a subset $\mathcal{D}\subset \mathcal{T}_A$, $\overrightarrow{\mathcal{D}}$ means the full subquiver of $\overrightarrow{\mathcal{K}_A}$ consisting of elements in $\mathcal{D}$.

Theorem \ref{The 1} provides a map $\Phi:\mathcal{T}_{\mathcal{T}\cup\mathcal{F}}\to \mathcal{T}_{B_0}$ (see Proposition \ref{injection} below). Restricting this map to certain subquivers of $\overrightarrow{\mathcal{K}_\Lambda}$ yields some quiver isomorphisms. More precisely, we obtain the following result.

\begin{Theorem}\label{The 2}
When $\Lambda$ is hereditary and $T_0$ is a BB-tilting $\Lambda$-module, there exists isomorphisms of quivers between $\overrightarrow{\mathcal{T}_{\mathcal{T}}}$ and $\overrightarrow{\mathcal{T}_{\mathcal{Y}}}$, $\overrightarrow{\mathcal{T}_{\mathcal{T},\mathcal{F}}}$ and $\overrightarrow{\mathcal{T}_{\mathcal{X},\mathcal{Y}}}$, respectively.
\end{Theorem}
\noindent With this conclusion and some further observations (see Propositions \ref{arrow}, \ref{arrow2}, \ref{arrow3}), we can construct $\overrightarrow{\mathcal{K}_{B_0}}$ from  $\overrightarrow{\mathcal{K}_{\Lambda}}$.

In Section 1, we recall some basic definitions and well-known results, and we will also state some Lemmas for future reference. In Section 2, we will prove Theorem \ref{The 1} and some related results. In Section 3, when $\Lambda$ is hereditary, $T_0$ is a BB-tilting $\Lambda$-module and $B_0=\End_\Lambda(T_0)$, we will prove Theorem \ref{The 2} and use it to construct $\overrightarrow{\mathcal{K}_{B_0}}$ from $\overrightarrow{\mathcal{K}_\Lambda}$. In section 4, we will give some examples to illustrate our results.

Throughout this paper, we suppose $A$ is a finite dimensional $k$ algebra over an algebraically closed field $k$ and assume that $A$ admits $n$ non-isomorphic simple modules. All modules considered here are finitely generated right modules and we use mod $A$ to denote the category of finitely generated right $A$-modules. For two homomorphisms $f:X\to Y$ and $g:Y\to Z$, their composition is written as $gf$. All subcategories are assumed to be full and closed under taking direct summmands. Given an $A$-module $M$,  $\mathsf{add} M$ means the subcategory of $A$-modules which are direct summands of direct sums of $M$, $\mathsf{fac} M$ means the images of $\mathsf{add} M$,  $\mathsf{Gen} M$ means the  subcategory of modules that are generated by $M$,  $\mathsf{Cogen} M$ means the  subcategory of modules that are cogenerated by $M$, $\mathsf{M}^{\perp}$ means the subcategory of $A$-modules $\mathsf{M^{\perp}}=\{N\in\text{mod }A|\Ext^{i}_A(M,N)=0 \text{ for }i\geq 1\}$, ${\rm pd}_AM$ means the projective dimension of $M$. $D: \text{mod }A\to \text{mod }A^{op}$ denotes the standard duality, $\Tr M$ means the transpose of $M$,  and $\tau M=D\Tr M$ denotes the Auslander-Reiten translation. Let $\{P(i)|i=1,2\ldots n\}$ be a complete set of non-isomorphic indecomposable projective $A$-modules and $P[i]=\bigoplus_{j\neq i}P(j)$ (indecomposable injective modules $I(i)$ and simple modules $S(i)$ are denoted similarly).

\section*{Acknowledgements}
The results of this paper are part of the author's doctoral thesis, which he is currently writing at University Kiel. The author thanks his advisor, Rolf Farnsteiner, for his continuous support and comments that helped a lot to improve the exposition of this paper.

\section{Preliminaries}

\begin{Definition}
An $A$-module $T$ is called \textbf{tilting} provided:
\begin{enumerate}
\item The projective dimension of $T$ is finite;
\item $\Ext^{i}_{A}(T,T)=0$ for all $i\geq 1$;
\item There exists an exact sequence: $0 \to A\to T_0 \to T_1 \to\cdots\to T_{r-1}\to T_r\to 0$ with $T_i\in \mathsf{add} T$ for $0\leq i\leq r$.
\end{enumerate}
\end{Definition}

In the following, modules satisfying the condition $(b)$ are called \textbf{self-orthogonal}, modules satisfying the conditions $(a)$ and $(b)$ are called \textbf{exceptional}, and we call a module \textbf{perfect exceptional} if is exceptional and admits $n$ non-isomorphic indecomposable summands. 

A tilting module $T_A$ is called \textbf{classic tilting} if ${\rm pd} _A(T)\leq 1$, and a module $M$ is called \textbf{(classic) partial tilting} if it is a direct summand of a (classic) tilting module.
Partial tilting modules with $n-1$ indecomposable summands will be called \textbf{almost complete tilting modules}.

\begin{Remark}
For the classic case (tilting modules with projective dimension no more than 1) tilting modules are the same as perfect exceptional, but whether this holds in general (tilting modules of finite projective dimension) is an open problem. We will see in Proposition \ref{perfecttilting},  for some special types of algebras, the answer is affirmative. 
As noted in the introduction of \cite{RS}, this follows from the proof of the main result of \cite{HRS}. But our proof is somewhat shorter and easier. 
Due to Bongartz's Lemma, $M$ is classic partial tilting if and only if ${\rm pd}M\leq 1$ and $\Ext^{1}_A(M,M)=0$. But this is not true in general, i.e., exceptional modules may be not partial tilting modules. We refer to  \cite{RS} for more details.
\end{Remark}

Given an $A$-module $T$, consider the full subcategories of mod $A$:
$$\mathcal{T}(T)=\{M\in\text{mod }A|\Ext^{1}_A(T,M)=0)\}$$and $$\mathcal{F}(T)=\{N\in\text{mod }A|\Hom_A(T,N)=0\}.$$
Let $B=\End_A(T)$, we obtain two full subcategories of mod $B$:
 $$\mathcal{X}(T)=\{X\in \text{mod }B|X\otimes_B T=0\}$$ and $$\mathcal{Y}(T)=\{Y\in\text{mod }B| \Tor^{B}_1(Y,T)=0\}.$$
When $T$ is classic tilting, according to \cite{KB}, \cite{HR}, $(\mathcal{T}(T),\mathcal{F}(T))$ and $(\mathcal{X}(T),\mathcal{Y}(T))$ are torsion pairs in mod-$A$ and mod-$B$, respectively. The connections between these two torsion pairs are stated in the  Tilting Theorem.

\begin{TheoremS}[Tilting Theorem]\cite[1.6]{KB}\label{Th0}
Let $A$ be an algebra, $T_A$ be a classic tilting module, $B=\End_A(T)$, and $(\mathcal{T}(T),\mathcal{F}(T))$, $(\mathcal{X}(T),\mathcal{Y}(T))$ be the induced torsion pairs in mod  $A$ and mod $B$, respectively. Then $T$ has the following properties:
\begin{enumerate}
\item The functors $\Hom_A(T,-)$ and $-\otimes_BT$ induce an quasi-inverse equivalence between $\mathcal{T}(T)$ and $\mathcal{Y}(T)$;
\item The functors $\Ext^{1}_A(T,-)$ and $\Tor^{B}_1(-,T)$ induce an quasi-inverse equivalence between $\mathcal{F}(T)$ and $\mathcal{X}(T)$.
\end{enumerate}
\end{TheoremS}

We will occasionally use the following Lemmas.
\begin{LemmaS}\cite[\uppercase\expandafter{\romannumeral 6}. Theorem 2.5]{ASS}\label{L0}
Let $T_A$ be a classic partial tilting module. The following statements are equivalent:
\begin{enumerate}
\item $T_A$ is a classic tilting module;
\item $\mathsf{Gen} T =\mathcal{T}(T)$;
\item $\mathsf{Cogen} \tau T=\mathcal{F}(T)$;
\item For every module $M\in\mathcal{T}(T)$, there exists a short exact sequence $0\to L\to T_0\to M\to 0$ with $T_0\in \mathsf{add} T$ and $L\in\mathcal{T}(T)$;
\item  For every module $N\in\mathcal{F}(T)$, there exists a short exact sequence $0\to N\to T_1\to K\to 0$ with $T_1\in \mathsf{add} \tau T$ and $K\in\mathcal{F}(T)$.
\end{enumerate}
\end{LemmaS}

\begin{LemmaS}\cite[\uppercase\expandafter{\romannumeral 6}. Lemma 3.2]{ASS}\label{L1}
Let $T_A$ be a classic tilting module, $B=\End_A(T)$, $M,N\in\mathcal{T}(T)$. Then we have functorial isomorphisms:
\begin{enumerate}
\item $\Hom_A(M,N)\cong \Hom_B(\Hom_A(T,M),\Hom_A(T,N))$;
\item $\Ext^{1}_A(M,N)\cong \Ext^{1}_B(\Hom_A(T,M),\Hom_A(T,N))$.
\end{enumerate}
\end{LemmaS}

The following Lemma is a special case of a dual version of Lemma \ref{L1}. For the reader's convenience, we write down the proof.

\begin{LemmaS}\label{L2}
Let $\Lambda$ be hereditary, $T_\Lambda$ be tilting, $B=\End_\Lambda(T)$, $M,N\in\mathcal{F}(T)$. Then we have functorial isomorphisms:
\begin{enumerate}
\item $\Hom_\Lambda(M,N)\cong \Hom_B(\Ext^{1}_\Lambda(T,M),\Ext^{1}_\Lambda(T,N))$;
\item $\Ext^{1}_\Lambda(M,N)\cong \Ext^{1}_B(\Ext^{1}_\Lambda(T,M),\Ext^{1}_\Lambda(T,N))$.
\end{enumerate}
\end{LemmaS}
\begin {proof}
The statement $(a)$ is a direct consequence of the Tilting Theorem \ref{Th0}, so it suffices  to show $(b)$. Since $N\in\mathcal{F}(T)$, Lemma \ref{L0} provides an exact sequence 
$$\xymatrix@1{0\ar[r]&N\ar[r]&T^{'}\ar[r]&N^{'}\ar[r]&0&(\dagger)
}$$
such that $T^{'}\in \mathsf{add}\tau T$, $N^{'}\in\mathcal{F}(T)$.
 Since $N^{'}\in\mathcal{F}(T)$ and  ${\rm pd}_\Lambda T\leq 1$, applying the functor $\Hom_\Lambda(T,-)$ to $(\dagger)$, we obtain an exact sequence of $B$-modules
$$\xymatrix@1{0\ar[r]&\Ext^{1}_\Lambda(T,N)\ar[r]&\Ext^{1}_\Lambda(T,T^{'})\ar[r]&\Ext^{1}_\Lambda(T,N^{'})\ar[r]&0.&(\dagger\dagger)
}$$
Since $\Ext^{1}_\Lambda(T,T^{'})$ is injective as $B$-module (\cite[\uppercase\expandafter{\romannumeral 6} Proposition 5.8]{ASS}), applying the functor $\Hom_B(\Ext^{1}_\Lambda(T,M),-)$ to $(\dagger\dagger)$ yields an exact sequence
$$\xymatrix@1{0\ar[r] &\Hom_B(\Ext^{1}_\Lambda(T,M),\Ext^{1}_\Lambda(T,N))\ar[r]& \Hom_B(\Ext^{1}_\Lambda(T,M),\Ext^{1}_\Lambda(T,T^{'}))\ar[r]&}$$
$$\xymatrix@1{ \Hom_B(\Ext^{1}_\Lambda(T,M),\Ext^{1}_\Lambda(T,N^{'}))\ar[r]& \Ext^{1}_B(\Ext^{1}_\Lambda(T,M),\Ext^{1}_\Lambda(T,N))\ar[r]& 0.}$$
The Auslander-Reiten formulas imply $\Ext^{1}_\Lambda(M,T^{'})\cong D\Hom_\Lambda(\tau^{-1}T^{'},M)=0$. Applying the functor $\Hom_\Lambda(M,-)$ to $(\dagger)$ we thus obtain 
$$\xymatrix@1{0\ar[r] &\Hom_\Lambda(M,N)\ar[r]& \Hom_\Lambda(M,T^{'})\ar[r]& \Hom_\Lambda(M,N^{'})\ar[r]& \Ext^{1}_\Lambda(M,N)\ar[r]& 0.}$$
Comparing it with the preceding exact sequence, the result follows the functorial isomorphism of $(a)$.
\end{proof}

\begin{LemmaS}\cite[\uppercase\expandafter{\romannumeral 6} Lemma 5.5]{ASS}\label{L3}
Let $T_A$ be a classic tilting module, $B=\End_A(T)$. If $M\in \mathcal{T}(T)$ and $N\in\mathcal{F}(T)$, then, for $j\geq 1$, there is an isomorphism
$$\Ext^{j}_A(M,N)\cong \Ext^{j-1}_B(\Hom_A(T,M),\Ext^{1}_A(T,N)).$$
\end{LemmaS}

The following Lemma is dual to Lemma \ref{L3}, we write down the proof for convenience.

\begin{LemmaS}\label{L4}
Let $T_A$ be a classic tilting module, $B=\End_A(T)$. If $M\in \mathcal{T}(T)$ and $N\in\mathcal{F}(T)$, then, for $j\geq 0$, there is an isomorphism
$$\Ext^{j+1}_B(\Ext^{1}_A(T,N),\Hom_A(T,M))\cong \Ext^{j}_A(N,M). $$
\end{LemmaS}
\begin{proof} 
Consider the short exact sequence of $B$-modules
$$\xymatrix@1{0\ar[r]&K^{'}\ar[r]&P^{'}\ar[r]&\Ext^{1}_A(T,N)\ar[r]&0,&(*)}$$
such that $P^{'}\to \Ext^{1}_A(T,N)$ is the projective cover of $\Ext^{1}_A(T,N)$. 
Since projective $B$-modules are torsion free and  torsion free modules are closed under submodules, $P^{'}$ and $K^{'}$ belong to $\mathcal{Y}(T)$. According to the Tilting Theorem \ref{Th0}, there exists $M^{'}\in\mathcal{T}(T)$ such that $K^{'}\cong \Hom_A(T,M^{'})$, and since $P^{'}$ is projective $P^{'}\cong \Hom_A(T,\overline{T})$ for some $\overline{T}\in \mathsf{add} T$ . Hence the exact sequence$(*)$ can be rewritten as 
$$\xymatrix@1{0\ar[r]&\Hom_A(T,M^{'})\ar[r]&\Hom_A(T,\overline{T})\ar[r]&\Ext^{1}_A(T,N)\ar[r]&0.&(**)}$$
Applying the functor $-\otimes T$ to $(**)$, \cite[1.6]{KB} yields an exact sequence of $A$-modules
$$0\to \Tor^{B}_1(\Ext^{1}_A(T,N),T)\to \Hom_A(T,M^{'})\otimes_B T\to \Hom_A(T,\overline{T})\otimes_BT\to 0.$$
According to \cite[1.6]{KB}, the preceding exact sequence is isomorphic to  
$$\xymatrix@1{0\ar[r]&N\ar[r]&M^{'}\ar[r]&\overline{T}\ar[r]&0.&(***)}$$
Since $\Ext^{1}_A(T,N)\in\mathcal{X}(T)$ and $\Hom_A(T,M)\in\mathcal{Y}(T)$, $\Hom_B(\Ext^{1}_A(T,N),\Hom_A(T,M))=0$.
Then since $\Hom_A(T,\overline{T})$ is a projective $B$-module, applying the functor $\Hom_B(-,\Hom_A(T,M))$ to $(**)$, we obtain an exact sequence 
$$\xymatrix@1{0\ar[r]&\Hom_B(\Hom(T,\overline{T}),\Hom_A(T,M))\ar[r]& \Hom_B(\Hom_A(T,M^{'}),\Hom_A(T,M))\ar[r]&}$$
$$\xymatrix@1{\Ext^{1}_B(\Ext^{1}_A(T,N),\Hom_A(T,M))\ar[r]& 0&(****)}$$
and isomorphisms $\Ext^{i}_B(\Hom_A(T,M^{'}),\Hom_A(T,M))\cong \Ext^{i+1}_B(\Ext^{1}_A(T,N),\Hom_A(T,M))(\ddagger)$
for $i\geq 1$. Applying the functor $\Hom_A(-,M)$ to $(***)$, we obtain a short exact sequence 
$$\xymatrix@1{0\ar[r]&\Hom_A(\overline{T},M)\ar[r]& \Hom_A(M^{'},M)\ar[r]& \Hom_A(N,M)\ar[r]& 0&(*****)}$$
and isomorphisms $\Ext^{i}_A(M^{'},M)\cong \Ext^{i}_A(N,M)(\ddagger\ddagger)$ for $i\geq 1$. Lemma \ref{L1}, applied to $\overline{T},M,M^{'}\in\mathcal{T}(T)$, shows that the first two terms of the sequences $(****)$ and $(*****)$ are isomorphic. As the isomorphisms are compatible with the sequences, we obtain $\Ext^{1}_B(\Ext^{1}_A(T,N),\Hom_A(T,M))\cong \Hom_A(N,M)$ as well as $\Ext^{i+1}_B(\Ext^{1}_A(T,N),\Hom_A(T,M))$ $\stackrel{(\ddagger)}\cong \Ext^{i}_B(\Hom_A(T,M^{'}),\Hom_A(T,M))$ $\cong \Ext^{i}_A(M^{'},M)\stackrel{(\ddagger\ddagger)}\cong \Ext^{i}_A(N,M)$ for $i\geq 1$.
\end{proof}

\begin{LemmaS}\label{L5}
Let $\xymatrix@1{0\ar[r]&X\ar[r]^{f}&Y\ar[r]^{g}&Z\ar[r]&0}$ be a short exact sequence of $A$-modules. Then the following statements hold:
\begin{enumerate}
\item If $X\ne 0$ and for each $\varphi:X\to X$, there exists $\phi: Y\to Y$ such that $\phi f=f \varphi$, then $Y$ being indecomposable implies $X$ is indecomposable.
\item If $Z\not=0$ and for each $\psi:Z\to Z$ there exists $\phi:Y\to Y$ such that $g\phi=\psi g$, then $Y$ being indecomposable implies $Z$ is indecomposable.
\item If $\Hom_A(Y,X)=\Ext^{1}_{A}(Y,X)=0$ and for each $\phi:Y\to Y$ there exists $\psi:Z\to Z$ such that $\psi g=g\phi$, then $Z$ being indecomposable implies $Y$ is indecomposable.
\end{enumerate}
\end{LemmaS}
\begin{proof}
\begin{enumerate}[fullwidth,itemindent=2em]
\item Let $\varphi\in \End_A(X)$. By assumption there exists $\phi: Y\to Y$ such that $\phi  f=f \varphi$, whence $\phi^{i} f=f \varphi^{i}$ for any $i\geq 1$. If $\phi$ is nilpotent, then there exists $i\geq 1$ such that $f\varphi^{i}=0$, so $\varphi$ is nilpotent. If $\phi$ is invertible, due to the Snake Lemma $\varphi$ is injective, and hence invertible. We conclude that $X$ is indecomposable.
\item the proof is similar to $(a)$.
\item For any $\phi\in \End_A(Y)$, by assumption there exists $\psi\in \End_A(Z)$ such that $\psi g=g\phi$, so $\psi^{i}g=g\phi^{i}$ for any $i\geq 1$. If $\psi$ is nilpotent, then there exists $i\geq 1$ such that $g\phi^{i}=0$, hence there exists $h:Y\to X$ such that $\phi^{i}=fh$. However, $\Hom_A(Y,X)=0$, so $\phi$ is nilpotent. If $\psi$ is invertible, consider the following commutative diagram:
$$\xymatrix{0\ar[r]&X\ar[r]^{f}\ar[d]^{\alpha}&Y\ar[r]^{g}\ar[d]^{\phi}&Z\ar[r]\ar[d]^{\psi}&0\\
0\ar[r]&X\ar[r]^{f}\ar@{-->}[d]^{\beta}&Y\ar[r]^{g}\ar@{-->}[d]^{\epsilon}&Z\ar[r]\ar[d]^{\psi^{-1}}&0\\
0\ar[r]&X\ar[r]^{f}&Y\ar[r]^{g}&Z\ar[r]&0.&
}$$
Where the existence of $\epsilon$ is ensured by $\Ext^{1}_A(Y,X)=0$. Then $g=g\epsilon\phi$, i.e., $g(\epsilon\phi-id_Y)=0$, so there exists $h^{'}:Y\to X$ such that $\epsilon\phi-id_Y=fh^{'}$. Since $\Hom_A(Y,X)=0$, we obtain $\epsilon\phi=id_Y$. Thus, $\phi$ is injective and hence invertible. Consequently, $Y$ is indecomposable.
\end{enumerate}
\end{proof}

\begin{LemmaS}\cite[Lemma 4.5]{HR}\label{L6}
Suppose $A$ is hereditary, if $T_1$ and $T_2$ are indecomposable $A$-modules such that $\Ext^{1}_A(T_2,T_1)=0$, then any nonzero homorphism from $T_1$ to $T_2$ is either a monomorphism or an  epimorphism.
\end{LemmaS}

\section{Construction of tilting modules of tilted algebras}
An algebra $\Lambda^{'}$ is called \textbf{tilted algebra} of type $\Lambda$,  if $\Lambda^{'}=\End_\Lambda(T)$ with $T_\Lambda$ being tilting over a hereditary algebra $\Lambda$. We mentioned in Section 1 that for hereditary algebras, being perfect exceptional is the same as being tilting, we will show firstly this is still true for tilted algebras. Let $Z=\bigoplus^{s}_{i=1}Z_i\in \text {mod }A$ be a basic classic partial tilting module with $s$ pairwise non-isomorphic indecomposable summands, the \textbf{right perpendicular subcategory} $\mathsf{Z^{perp}}$ is defined to be the full subcategory $\mathsf{Z^{perp}}=\{M\in \text{mod } A|\Hom_A(Z,M)=\Ext^{1}_A(Z,M)=0\}$. Then according to \cite{GL}  or \cite[Theorem 2.1]{HaUn1}, we have the following Theorem. Recall that $n$ is the number of isomorphism classes of simple $A$-modules.

\begin{TheoremS}\label{perpendicular}
The perpendicular category $\mathsf{Z^{perp}}$ contains a projective generator $Q$, such that $\mathsf{Z^{perp}}\cong \text{mod }\Gamma$, where $\Gamma=\End_A(Q)$ is a hereditary finite dimensional $k$-algebra, and  the number of non-isomorphic simple $\Gamma$-modules equals $n-s$. Moreover, $\Ext^{i}_A(M,N)\cong \Ext^{i}_\Gamma(M,N)$ for all $M,N\in \mathsf{Z^{perp}}$ and $i\geq 0$.
\end{TheoremS}

The relationship between exceptional modules and tilting modules is stated as the following Lemma.

\begin{LemmaS}\cite[Proposition 3.6]{Ha2}\label{L7}
The following statements are equivalent for an exceptional $A$-module $N$:
\begin{enumerate}
\item $N$ is tilting;
\item $\mathsf{N^{\perp}}\subset \mathsf{fac}N$.
\end{enumerate}
\end{LemmaS}

The following result also follows from the proof of the main result of \cite{HRS}. We provide a shorter proof that uses the techniques  presented in the foregoing section.

\begin{PropositionS}\label{perfecttilting}
Let $\Lambda$ be hereditary , $T_0$ be a tilting $\Lambda$-module, $B_0=\End_\Lambda(T_0)$. Then if $T$ is a perfect exceptional $B_0$-module, $T$ is tilting.
\end{PropositionS} 
\begin{proof}
According to \cite[1.7]{KB}, the torsion pair $(\mathcal{X}(T_0),\mathcal{Y}(T_0))$ is splitting in mod $B_0$. So without loss of generality, we assume $T=Y^{'}\oplus X^{'}$ such that $Y^{'}\in\mathcal{Y}(T_0)$, $X^{'}\in\mathcal{X  }(T_0)$. Then the Tilting Theorem \ref{Th0} provides $Y\in\mathcal{T}(T_0)$ and $X\in\mathcal{F}(T_0)$ such that $Y^{'}\cong \Hom_\Lambda(T_0,Y)$, $X^{'}\cong \Ext^{1}_\Lambda(T_0,X)$.
In view of Theorem \ref{perpendicular}, there is $\Gamma$ such that $\text{mod }\Gamma=\mathsf{X^{perp}}$, we claim $Y$ is a tilting $\Gamma$-module. By Lemma \ref{L4}, $\Ext^{i}_\Lambda(X,Y)\cong \Ext^{i+1}_{B_0}(\Ext^{1}_\Lambda(T_0,X),\Hom_\Lambda$ $(T_0,Y))$ for $i=0,1$. Since $T$ is perfect exceptional, $\Ext^{i}_\Lambda(X,Y)$ $=0$ (\ding{172}) for $i=0,1$, i.e., $Y$ in $\mathsf{X^{perp}}$. Then Theorem \ref{perpendicular} implies $Y$ is perfect exceptional and hence a  tilting $\Gamma$-module.

According to Lemma \ref{L7}, to show $T$ is tilting, it suffices to verify  $Q\in \mathsf{fac} T$ for every indecomposable $B_0$-module $Q$ which belongs to $\mathsf{T^{\perp}}$. 
Without loss of generality, let $Q\in \mathsf{T^{\perp}}$ be an indecomposable $B_0$-module. Since $(\mathcal{X}(T_0),\mathcal{Y}(T_0))$ is splitting in mod $B_0$, $Q$ is either in $\mathcal{X}(T_0)$ or in $\mathcal{Y}(T_0)$.

First, we assume $Q\in\mathcal{Y}(T_0)$. According to the Tilting Theorem \ref{Th0}, there exists $L\in\mathcal{T}(T_0)$ such that $Q\cong \Hom_\Lambda(T_0,L)$. By Lemma \ref{L1}, $\Ext^{1}_\Lambda(Y,L)\cong \Ext^{1}_{B_0}(\Hom_\Lambda(T_0,Y),\Hom_\Lambda(T_0,L))$. As $Q\in \mathsf{T^{\perp}}$, we thus have $\Ext^{1}_\Lambda(Y,L)=0$. Again, $Q\in \mathsf{T^{\perp}}$ implies $\Ext^{i}_{B_0}(\Ext^{1}_\Lambda(T_0,X),\Hom_\Lambda(T_0,L))$ $=0$ for $i=1,2$. Then Lemma \ref{L4} shows that $L\in \mathsf{X^{perp}}=\text{mod }\Gamma$. Since $Y$ is a tilting $\Gamma$-module, Theorem \ref{perpendicular} implies $\Ext^{1}_\Gamma(Y,L)\cong \Ext^{1}_\Lambda(Y,L)=0$. By the same token, Lemma \ref{L0} provides an exact sequence:
$$\xymatrix@1{0\ar[r]&Z\ar[r]&\overline{Y}\ar[r]& L\ar[r]& 0&(*)}$$
in mod $\Gamma$ and also in mod $\Lambda$ such that $\overline{Y}\in \mathsf{add} Y$ and $Z\in\mathcal{T}(Y)=\mathsf{Gen} Y$ (in mod $\Gamma$ and hence in mod $\Lambda$). Since $Y\in\mathcal{T}(T_0)$ and the torsion class is closed under taking factor modules,  $Z\in\mathcal{T}(T_0)$.  Applying the functor $\Hom_\Lambda(T_0,-)$ to $(*)$ we thus obtain $Q\cong \Hom_\Lambda(T_0,L)\in \mathsf{fac}\Hom_\Lambda(T_0,Y)\subset \mathsf{fac} T$.

Next, we assume $Q\in\mathcal{X}(T_0)$. According to Theorem \ref{Th0}, there exists $L\in\mathcal{F}(T_0)$ such that $Q\cong \Ext^{1}_\Lambda(T_0,L)$. By Lemma \ref{L2}, $ \Ext^{1}_\Lambda(X,L)\cong \Ext^{1}_{B_0}(\Ext^{1}_\Lambda(T_0,X),\Ext^{1}_\Lambda(T_0,L))$, so $Q\in \mathsf{T^{\perp}}$ implies $\Ext^{1}_\Lambda(X,L)=0$ (\ding{173}). Since $X$ is classic partial tilting, \cite[\uppercase\expandafter{\romannumeral 6}. Lemma 2.3]{ASS} implies $(\mathsf{Gen} X,\mathcal{F}(X))$ is a torsion pair in mod $\Lambda$. Consider the canonical sequence of $L$ associated to the torsion pair $(\mathsf{Gen} X,\mathcal{F}(X))$: 
$$\xymatrix@1{0\ar[r] &L^{'}\ar[r]& L\ar[r] &L^{''}\ar[r]& 0.&(**)}$$
Identity \ding{173} implies $\mathsf{Ext}^{1}_\Lambda(X,L^{''})=0$, and since $L^{''}\in\mathcal{F}(X)$,  $L^{''}\in \mathsf{{X}^{perp}}$.

If $L\in\mathcal{T}(Y)$, then $L^{''}\in\mathcal{T}(Y)$, so that Theorem \ref{perpendicular} yields $\Ext^{1}_\Gamma(Y,L^{''})\cong \Ext^{1}_\Lambda(Y,L^{''})=0$. Since $Y$ is a tilting $\Gamma$-module, Lemma \ref{L0} implies that $L^{''}$ is generated by $Y$ (in mod $\Gamma$ and hence in mod $\Lambda$), hence in $\mathcal{T}(T_0)$, i.e., $\Ext^{1}_\Lambda(T_0,L^{''})=0$. Since $L^{'}\in \mathsf{Gen} X$, applying the functor $\Hom_\Lambda(T_0,-)$ to $(**)$, we obtain $Q\cong \Ext^{1}_\Lambda(T_0,L)\in \mathsf{fac} \Ext^{1}_\Lambda(T_0,L^{'})\subset \mathsf{fac} \Ext^{1}_\Lambda(T_0,X)\subset \mathsf{fac} T$.

If $L\notin\mathcal{T}(Y)$, we have $\dim_k\Ext^{1}_\Lambda(Y,L)=:s\not=0$. Consider the following commutative diagram with exact rows and columns:
$$\xymatrix{
&0\ar[d]&0\ar[d]\\
0\ar[r]&L_1\ar[r]^{i}\ar[d]_{h}&F_1\ar[d]^{h^{'}}\ar@{-->}[dl]^{u}\\
0\ar[r]&L\ar[r]_{f}\ar[d]_{k}&F\ar[r]_{g}\ar[d]^{l}&Y^{s}\ar[r]&0\\
0\ar[r]&L_2\ar[r]\ar[d]&F_2\ar[d]\\
&0&0,
}
$$
where the construction of the middle row is similar to the proof of the Bongartz's Lemma (\cite[Lemma 2.1]{KB}): The connecting morphism $\delta:\Hom_\Lambda(Y,Y^{s})\to \Ext^{1}_\Lambda(Y,L)$ is surjective.
The two columns of short exact sequences are the canonical sequences of $L$ and $F$ associated to the torsion pair $(\mathsf{Gen} X,\mathcal{F}(X))$, respectively.  The existence of $i$ is due to the universal properties of canonical sequences of torsion pairs and is obviously a monomorphism. In view of $F_1\in \mathsf{Gen} X$ and \ding{172}, we have $\Hom_\Lambda(F_1,Y^{s})=0$, whence $gh^{'}=0$. So there exists $u:F_1\to L$ such that $h^{'}=fu$. Since $h^{'}$ is a monomorphism, $u$ is a monomorphism. By definition, $L_1=\sum Im\beta\{\beta:X\to L\}$, and $F_1\in \mathsf{Gen} X$ yields $\dim_kF_1\leq \dim_kL_1$. Hence $i$ is an isomorphism.
Owing to \ding{172} and \ding{173}, an application of functor $\Hom_\Lambda(X,-)$ to the middle row gives $\Ext^{1}_\Lambda(X,F)=0$ and hence $\Ext^{1}_\Lambda(X,F_2)=0$. By the construction of the second column, $F_2\in\mathcal{F}(X)$, i.e.,   $\Hom_\Lambda(X, F_2)=0$, so $F_2\in \mathsf{X^{perp}}$. By the construction of the middle row , we have $\Ext^{1}_\Lambda(Y,F)=0$, so $\Ext^{1}_\Lambda(Y,F_2)=0$. Hence  Theorem \ref{perpendicular} yields $\Ext^{1}_\Gamma(Y,F_2)\cong \Ext^{1}_\Lambda(Y,F_2)=0$. Since $Y$ is tilting in $\mathsf{X^{perp}}$, Lemma \ref{L0} implies $F_2$ is generated by $Y$ (in mod $\Gamma$ and hence in mod $\Lambda$), hence in $\mathcal{T}(T_0)$. Consequently, $\Ext^{1}_\Lambda(T_0,h^{'})$ is an epimorphism. Since $fh=h^{'}i$, $\Ext^{1}_\Lambda(T_0,fh)$ is also an epimorphism. Applying the functor $\Hom_\Lambda(T_0,-)$ to the first column and the middle row  we obtain the following diagram with exact rows:
$$\xymatrix{ \Hom_\Lambda(T_0,F)\ar[r]& \Hom_\Lambda(T_0,Y^{s})\ar[r]^{\alpha}& \Ext^{1}_\Lambda(T_0,L)\ar[r]^{\Ext^{1}_\Lambda(T_0,f)}&\Ext^{1}_\Lambda(T_0,F)\\
 \Hom_\Lambda(T_0,L_2)\ar[r]& \Ext^{1}_\Lambda( T_0,L_1)\ar[r]^{\Ext^{1}_\Lambda(T_0,h)}& \Ext^{1}_\Lambda(T_0,L)\ar[r]\ar[u]^{id}& \Ext^{1}_\Lambda(T_0,L_2).}
$$
For any $\eta\in \Ext^{1}_\Lambda(T_0,L)$, there exists $\xi\in \Ext^{1}_\Lambda(T_0,L_1)$ such that $\Ext^{1}_\Lambda(T_0,f)(\eta)=\Ext^{1}_\Lambda(T_0,fh)$ $(\xi)=\Ext^{1}_\Lambda(T_0,f)(\Ext^{1}_\Lambda(T_0,h)(\xi))$, i.e., $\Ext^{1}_\Lambda(T_0,f)(\eta-\Ext^{1}_\Lambda(T_0,h)(\xi))=0$. So there exists $v\in \Hom_\Lambda(T_0,Y^{s})$ such that $\alpha (v)=\eta-\Ext_\Lambda^{1}(T_0,h)(\xi)$, i.e., $\eta=\alpha(v)+\Ext_\Lambda^{1}(T_0,h)(\xi)$. Consequently, the homomorphism $(\alpha, \Ext^{1}_\Lambda(T_0,h)):\Hom_\Lambda(T_0,Y^{s})\oplus \Ext^{1}_\Lambda(T_0,L_1)\to \Ext^{1}_\Lambda(T_0,L)$ is an epimorphism, i.e., $Q\cong \Ext^{1}_\Lambda(T_0,L)\in \mathsf{fac}(\Hom_\Lambda(T_0,Y)\oplus \Ext^{1}_\Lambda(T_0,L_1))$. But since $L_1\in \mathsf{Gen} X$, this entails $Q\in \mathsf{fac}T$. This completes the proof.
\end{proof}

With this Proposition in hand, we turn to the construction of perfect exceptional modules for tilted algebras. Let us recall the definition of the trace. Let $\mathcal{U}$ be a class of $A$-modules and $M$ be an $A$-module, the \textbf{trace} of $\mathcal{U}$ in $M$ is defined to be
$$Tr_M(\mathcal{U})=\sum\{Imh|h:U\to M\text{ for some }U\in\mathcal{U}\}.$$
In particular, when $\mathcal{U}=\{U\}$, we will write $Tr_MU$.

\textbf{General assumption}: From now on, we let $\Lambda$ be hereditary, $T_0$ be a tilting $\Lambda$-module, $B_0=\End_\Lambda(T_0)$. Let $T=Y_0\oplus X_0$ be a basic tilting $\Lambda$-module such that  $Y_0\in\mathcal{T}(T_0)$, $X_0\in\mathcal{F}(T_0)$.
We will refer to this set-up by saying that ($\Lambda$, $T_0$, $B_0$, $Y_0$, $X_0$) fulfills the general assumption.

Consider the following short exact sequence 
$$\xymatrix@1{0\ar[r]& Tr_{Y_0}X_0\ar[r]^-{l}& Y_0\ar[r]^-{p}& Y_0/Tr_{Y_0}X_0\ar[r]& 0&(a)}$$
such that $l$ is the embedding and $(p,Y_0/Tr_{Y_0}X_0)$ is the cokernel. Obviously, $Tr_{Y_0}X_0\in \mathsf{Gen} X_0$,  hence $\Ext^{1}_\Lambda(X_0,Tr_{Y_0}X_0)=0$. By the definition of the trace, $\Hom_\Lambda(X_0,l):\Hom_\Lambda(X_0,Tr_{Y_0}X_0)\to \Hom_\Lambda(X_0,Y_0)$ is surjective and  $\Hom_\Lambda(X_0,Y_0/Tr_{Y_0}X_0)=0$ (\ding{174}).

\begin{LemmaS}\label{selforthogonal}
Suppose that ($\Lambda$, $T_0$, $B_0$, $Y_0$, $X_0$) fulfills the general assumption.  Then $\Ext^{1}_{B_0}(\Hom_\Lambda$ $(T_0,Y_0/Tr_{Y_0}X_0),\Hom_\Lambda(T_0,$ $Y_0/Tr_{Y_0}X_0))=0$.
\end{LemmaS}
\begin{proof}
 Since  $Y_0\in\mathcal{T}(T_0)$ and $\mathcal{T}(T_0)$ is  closed under factor modules, we have $Y_0/Tr_{Y_0}X_0\in\mathcal{T}(T_0)$. By Lemma \ref{L1}, we have $\Ext^{1}_{B_0}(\Hom_\Lambda(T_0,Y_0/Tr_{Y_0}X_0),\Hom_\Lambda(T_0,Y_0/Tr_{Y_0}X_0))\cong \Ext^{1}_\Lambda(Y_0/Tr_{Y_0}X_0,$ $Y_0/Tr_{Y_0}X_0)$, so we just need to show $\Ext^{1}_{\Lambda}(Y_0/Tr_{Y_0}X_0,Y_0/Tr_{Y_0}X_0)=0$. Applying the functor $\Hom_\Lambda(-,$ $Y_0/Tr_{Y_0}X_0)$ to $(a)$ yields an exact sequence 
 $$\xymatrix@1{\Hom_\Lambda(Tr_{Y_0}X_0,Y_0/Tr_{Y_0}X_0)\ar[r] &\Ext^{1}_\Lambda(Y_0/Tr_{Y_0}X_0,Y_0/Tr_{Y_0}X_0)\ar[r]& \Ext^{1}_\Lambda(Y_0,Y_0/Tr_{Y_0}X_0).&(b)}$$
In view of  $Tr_{Y_0}X_0\in \mathsf{Gen} X_0$ and \ding{174}, the left exactness of $\Hom_\Lambda(-,Y_0/Tr_{Y_0}X_0)$ forces $\Hom_\Lambda(Tr_{Y_0}X_0,$ $Y_0/Tr_{Y_0}X_0)=0$ (\ding{175}). Since $Y_0$ is partial tilting, an application of the functor $\Hom_\Lambda(Y_0,-)$ to $(a)$ yields $\Ext^{1}_\Lambda(Y_0,Y_0/Tr_{Y_0}X_0)=0$. Thus, we have $\Ext^{1}_\Lambda(Y_0/Tr_{Y_0}X_0,Y_0/Tr_{Y_0}X_0)=0$ (\ding{176}).
\end{proof}

\begin{LemmaS}\label{basic}
Assume that ($\Lambda$, $T_0$, $B_0$, $T$, $Y_0$, $X_0$) fulfills the general assumption. Then $Y_0/Tr_{Y_0}X_0$ is basic.
\end{LemmaS}
\begin{proof}
Suppose $Y_0=\bigoplus^{t}_{i=1}Y_i$ is the decomposition of $Y_0$ with indecomposable summands, for each $i$ we can consider the following exact sequence:
$$\xymatrix@1{
0\ar[r]& Tr_{Y_i}X_0\ar[r]^-{l_i}& Y_i\ar[r]^-{p_i}& Y_i/Tr_{Y_i}X_0\ar[r]& 0.&(a_i)}$$
According to \cite[ Proposition 8.18]{AF}, we know that $Tr_{Y_0}X_0=\bigoplus_{i=1}^{t} Tr_{Y_i}X_0$, and $Y_0/Tr_{Y_0}X_0=\bigoplus_{i=1}^{t}Y_i/Tr_{Y_i}X_0$.
Since $T=Y_0\oplus X_0$ is tilting and $Tr_{Y_0}X_0\in \mathsf{Gen} X_0$, by right exactness of $\Ext^{1}_\Lambda(Y_0,-)$, we obtain $\Ext^{1}_\Lambda(Y_0,Tr_{Y_0}X_0)=0$ (\ding{177}) and hence $\Ext^{1}_\Lambda(Y_i,Tr_{Y_i}X_0)=0$.
Application of the functor $\Hom_\Lambda(Y_i,-)$ to $(a_i)$ now shows that the homomorphism $\Hom_\Lambda(Y_i,p_i): \Hom_\Lambda(Y_i,Y_i)\to \Hom_\Lambda(Y_i,Y_i/Tr_{Y_i}X_0)$ is an epimorphism. 
So for each $\phi:Y_i/Tr_{Y_i}X_0\to Y_i/Tr_{Y_i}X_0$ there exists $\psi:Y_i\to Y_i$ such that $\phi p_i=p_i \psi$. By Lemma \ref{L5} $(b)$, we  know $Y_i/Tr_{Y_i}X_0$ is indecomposable or zero.

Suppose that there exists an isomorphism $\nu:Y_i/Tr_{Y_i}X_0\to Y_j/Tr_{Y_j}X_0$ with $Y_i/Tr_{Y_i}X_0\not= 0,Y_j/Tr_{Y_j}X_0\not= 0$ and $i\not=j$. Consider the following commutative diagram:
$$\xymatrix{0\ar[r]&Tr_{Y_i}X_0\ar[r]\ar@{-->}[d]^{h}&Y_i\ar[r]^-{p_i}\ar@{-->}[d]^{k}&Y_i/Tr_{Y_i}X_0\ar[d]^{\nu}\ar[r]&0&(1)\\
0\ar[r]&Tr_{Y_j}X_0\ar[r]&Y_j\ar[r]^-{p_j}&Y_j/Tr_{Y_j}X_0\ar[r]&0,&(2)
}$$
where the existence of $k$ is ensured by \ding{177}. According to Lemma \ref{L6}, $k$ is either a monomorphism or an epimorphism.

If $k$ is a monomorphism, consider the following commutative diagram obtained by using the Snake Lemma and taking the cokernels of $h$ and $k$ :
$$\xymatrix{0\ar[r]&Y_i\ar[r]&Y_j\ar[r]&E\ar[r]&0&(3)\\
0\ar[r]&Tr_{Y_i}X_0\ar[r]\ar[u]&Tr_{Y_j}X_0\ar[r]\ar[u]&E\ar[r]\ar[u]&0.&(4)}
$$
Applying $\Hom_\Lambda(-,Y_i/Tr_{Y_i}X_0)$ to $(4)$, we obtain an exact sequence 
$$\xymatrix@1{\Hom_\Lambda(Tr_{Y_i}X_0,Y_i/Tr_{Y_i}X_0)\ar[r]& \Ext^{1}_\Lambda(E,Y_i/Tr_{Y_i}X_0)\ar[r]& \Ext^{1}_\Lambda(Tr_{Y_j}X_0,Y_i/Tr_{Y_i}X_0).}$$
In view of the exact sequence $(b)$ and \ding{176}, we have $\Ext^{1}_\Lambda(Tr_{Y_0}X_0,Y_0/Tr_{Y_0}X_0)=0$. Then \ding{175} implies $\Ext^{1}_\Lambda(E,$ $Y_i/Tr_{Y_i}X_0)=0$ (\ding{178}). Applying the functor $\Hom_\Lambda(-,Y_i/Tr_{Y_i}X_0)$ to $(1)$, \ding{175} yields an isomorphism $\Hom_\Lambda(Y_i/Tr_{Y_i}X_0,Y_i/Tr_{Y_i}X_0)\cong \Hom_\Lambda(Y_i,Y_i/Tr_{Y_i}X_0)$. Thanks to Lemma \ref{L6} and \ding{176},  $\dim_k\Hom_\Lambda(Y_i,Y_i/Tr_{Y_i}X_0)=\dim_k\Hom_\Lambda(Y_i/Tr_{Y_i}X_0,Y_i/Tr_{Y_i}X_0)=1$. Applying the functor $\Hom_\Lambda(Y_i,-)$ to $(1)$ while observing  \ding{177}, we obtain an exact sequence 
$$\xymatrix@1{0\ar[r]& \Hom_\Lambda(Y_i,Tr_{Y_i}X_0)\ar[r]& \Hom_\Lambda(Y_i,Y_i)\ar[r]& \Hom_\Lambda(Y_i,Y_i/Tr_{Y_i}X_0)\ar[r]& 0.}$$ 
Obviously, $Y_i$ is self-orthogonal, by Lemma \ref{L6},  $\dim_k\Hom_\Lambda(Y_i,Y_i)=1$. Then $\dim_k\Hom_\Lambda(Y_i,$ $Y_i/Tr_{Y_i}X_0)=1$ implies $\dim_k\Hom_\Lambda$ $(Y_i,Tr_{Y_i}X_0)=0$, and hence $\Hom_\Lambda(Y_i,Tr_{Y_i}X_0)=0$ (\ding{179}).
Applying the functor $\Hom_\Lambda(-,Tr_{Y_i}X_0)$ to $(3)$, we obtain an exact sequence: 
$$\xymatrix@1{\Hom_\Lambda(Y_i,Tr_{Y_i}X_0)\ar[r]&\Ext^{1}_\Lambda(E,Tr_{Y_i}X_0)\ar[r]& \Ext^{1}_\Lambda(Y_j,Tr_{Y_i}X_0).}$$
Then \ding{177} implies $\Ext^{1}_\Lambda(E,Tr_{Y_i}X_0)=0$.
Upon applying the functor $\Hom_\Lambda(E,-)$ to $(1)$, \ding{178} forces $\Ext^{1}_\Lambda(E,Y_i)=0$ which implies $(3)$ is split. But since $Y_j$ is indecomposable and $Y_i\not=0$, we arrive at $E=0$. So $Y_i\cong Y_j$, a contradiction.

If $k$ is an epimorphism, consider the following commutative diagram obtained by using the Snake Lemma and taking the kernels of $h$ and $k$
$$\xymatrix{0\ar[r]&F\ar[r]&Y_i\ar[r]&Y_j\ar[r]&0&(5)\\
0\ar[r]&F\ar[r]\ar[u]&Tr_{Y_i}X_0\ar[r]\ar[u]&Tr_{Y_j}X_0\ar[r]\ar[u]&0.&(6)
}$$
Applying the functor $\Hom_\Lambda(Y_j,-)$ to $(6)$, yields an exact sequence 
$$\xymatrix{\Hom_\Lambda(Y_j,Tr_{Y_j}X_0)\ar[r] &\Ext^{1}_\Lambda(Y_j,F)\ar[r] &\Ext^{1}_\Lambda(Y_j,Tr_{Y_i}X_0).}$$
Now \ding{179} and  \ding{177} force $\Ext^{1}_\Lambda(Y_j,F)=0$ which implies $(5)$ is split, a contradiction. Hence we obtain the assertion.
\end{proof}

Now we are in a position to verify Theorem \ref{The 1}.

\begin{TheoremS}\label{perfect}
Suppose $(\Lambda$, $T_0$, $B_0$, $T$, $Y_0$, $X_0)$ satisfies the general assumption. If no indecomposable summand of $Y_0$ is generated by $X_0$, then $\Hom_\Lambda(T_0,Y_0/Tr_{Y_0}X_0)\oplus \Ext^{1}_\Lambda(T_0,X_0)$ is a tilting $B_0$-module.
\end{TheoremS}
\begin{proof}
By \cite[1.7]{KB}, we have ${\rm pd}_{B_0}\Hom_\Lambda(T_0,Y_0/Tr_{Y_0}X_0)\leq {\rm pd}_\Lambda Y_0/Tr_{Y_0}X_0\leq 1$ and ${\rm id}_{B_0} \Ext^{1}_\Lambda(T_0,$ $X_0)\leq {\rm id}_\Lambda X_0\leq 1$, so that $\Ext^{2}_{B_0}(\Hom_\Lambda(T_0,Y_0/Tr_{Y_0}X_0),\Hom_\Lambda(T_0,Y_0/Tr_{Y_0}X_0))=\Ext^{2}_{B_0}(\Ext^{1}_\Lambda$ $(T_0,X_0),\Ext^{1}_\Lambda(T_0,X_0))=0$. Lemma \ref{L3} implies $\Ext^{1}_{B_0}(\Hom_\Lambda(T_0,Y_0/Tr_{Y_0}X_0),$ $\Ext^{1}_\Lambda(T_0,X_0))\cong \Ext^{2}_\Lambda(Y_0/Tr_{Y_0}X_0,X_0)=0$.
Applying the functor $\Hom_\Lambda(X_0,-)$ to the exact sequence $(a)$, we obtain $\Ext^{1}_\Lambda(X_0,Y_0/Tr_{Y_0}X_0)=0$.
Then, by  Lemma \ref{L4}, $\Ext^{2}_{B_0}(\Ext^{1}_\Lambda(T_0,X_0),$ $\Hom_\Lambda(T_0,Y_0/Tr_{Y_0}X_0))\cong \Ext^{1}_\Lambda(X_0,Y_0/Tr_{Y_0}X_0)=0$, and $\Ext^{1}_{B_0}(\Ext^{1}_\Lambda(T_0,X_0),\Hom_\Lambda(T_0,Y_0/Tr_{Y_0}X_0))\cong \Hom_\Lambda(X_0,Y_0/$ $Tr_{Y_0}X_0)=0$. So $\Ext^{1}_\Lambda(T_0,X_0)\oplus \Hom_\Lambda(T_0,Y_0/Tr_{Y_0}X_0) $ is self-orthogonal and obviously exceptional. Note that $Y_i$ is generated by $X_0$ whenever $Y_i/Tr_{Y_i}X_0=0$, then since no indecomposable summand of $Y_0$ is generated by $X_0$, we obtain $Y_i/Tr_{Y_i}X_0\not=0$ for $i=1,2\cdots t$. According to \cite[\uppercase\expandafter{\romannumeral 3} 1.5]{Ha1},  $B_0$ admits the same number of non-isomorphic simple modules as $\Lambda$. Then by the Tilting Theorem \ref{Th0} and Lemmas \ref{selforthogonal}, \ref{basic}, we know it is perfect exceptional. By Proposition \ref{perfecttilting}, we know it's tilting and the conclusion follows.
\end{proof}

Thus, for each given tilting $\Lambda$-module $Y_0\oplus X_0$ with $Y_0\in\mathcal{T}(T_0)$ and $X_0\in\mathcal{F}(T_0)$, if no indecomposable summand of $Y_0$ is generated by $X_0$ , we obtain a tilting $B_0$-module $\Ext^{1}_\Lambda(T_0,X_0)\oplus \Hom_\Lambda(T_0,Y_0/Tr_{Y_0}X_0)$.

Similarly, we can also use the reject to construct tilting modules. Let $\mathcal{U}$ be a class of $A$-modules and $M$ be an $A$-module, the \textbf{reject} of $\mathcal{U}$ in $M$ is defined to be 
$$Rej_M(\mathcal{U})=\bigcap\{kerh|h:M\to U\text{ for some }U\in\mathcal{U}\}.$$
In particular, when $\mathcal{U}=\{U\}$, we just write $Rej_MU$.

Let's consider the following short exact sequence:
$$\xymatrix@1{0\ar[r]&Rej_{X_0}Y_0\ar[r]^-{\pi}&X_0\ar[r]^-{q}&X_0/Rej_{X_0}Y_0\ar[r]&0,&(a^{'})
}$$
such that $\pi$ is the embedding and $(q,Z_0)$ is the cokernel.

We will give the following conclusions but omit the proof (the reader is referred to\cite{HWP} for details).

\begin{CorollaryS}\label{Tiltingdual1}
Let $\Lambda$ be hereditary, $T_0$ be a tilting $\Lambda$-module, $T=Y_0\oplus X_0$  be a tilting $\Lambda$-module such that $Y_0\in\mathcal{T}(T_0)$, $X_0\in\mathcal{F}(T_0)$. If no  indecomposable summand of $X_0$ is cogenerated by $Y_0$, then  
 $\Hom_\Lambda(T_0,Y_0)\oplus \Ext^{1}_\Lambda(T_0,Rej_{X_0}Y_0)$ is a tilting $B_0$-module.
\end{CorollaryS}

\begin{CorollaryS}\label{Tiltingdual2}
Let $\Lambda$ be hereditary, $T_0$ be a tilting $\Lambda$-module, $T=Y_0\oplus X_0$  be a tilting $\Lambda$-module such that $Y_0\in\mathcal{T}(T_0)$, $X_0\in\mathcal{F}(T_0)$. If no indecomposable summand of $Y_0$ is generated by $X_0$ and no  indecomposable summands of $X_0$ is cogenerated by $Y_0$, then  
 $\Hom_\Lambda(T_0,Y_0/Tr_{Y_0}X_0)\oplus \Ext^{1}_\Lambda(T_0,Rej_{X_0}Y_0)$ is a tilting $B_0$-module.
\end{CorollaryS}

\section{Tilting quivers of BB-tilted algebras}
\noindent 
Recall that for a given algebra $A$, $\mathcal{T}_A$ denotes the set of equivalence classes of tilting $A$-modules. Two tilting modules $T_1,T_2$ belong to  the same class provided $\mathsf{add}(T_1)=\mathsf{add}(T_2)$. So when we consider elements of $\mathcal{T}_A$, we just consider basic tilting modules.

\begin{Remark}\label{ad}
For a given algebra $A$, the number of non-isomporphic indecomposable summands of a tilting module is constant. So if $T$ is  a tilting module and $X$ is exceptional such that $\Ext^{i}_A(T,X)=\Ext^{i}_A(X,T)=0$ for $i\geq 1$, then $T\oplus X$ is tilting and $\mathsf{add} T=\mathsf{add} (T\oplus X)$.
\end{Remark}

Recall that $\Lambda$ denotes a hereditary algebra and for a given a tilting $\Lambda$-module $T_0$,  $\mathcal{T}_{\mathcal{T}\cup\mathcal{F}}\subset\mathcal{T}_\Lambda$ is the subset of $\mathcal{T}_\Lambda$ consisting of equivalence classes whose representatives have indecomposable direct summands only in $\mathcal{T}(T_0)\cup\mathcal{F}(T_0)$.

We say that $T_0$ is \textbf{admissible}, if for every tilting $\Lambda$-module $T=Y_0\oplus X_0$ with $Y_0\in\mathcal{T}(T_0)$ and $X_0\in\mathcal{F}(T_0)$, no indecomposable summand of $Y_0$ is generated by $X_0$.

We will see in Lemma \ref{Simple}, BB-tilting modules are admissible. And examples will show there exists admissible titling modules which are not BB-tilting.

\begin{PropositionS}\label{injection}
Let $T_0$, $B_0$, $\mathcal{T}_{\mathcal{T}\cup\mathcal{F}}$ be as above with $T_0$ being admissible. Then  $$\Phi:\mathcal{T}_{\mathcal{T}\cup\mathcal{F}}\to \mathcal{T}_{B_0}$$ via $T=Y_0\oplus X_0\mapsto \Hom_\Lambda(T_0,Y_0/Tr_{Y_0}X_0)\oplus \Ext^{1}_\Lambda(T_0,X_0)$, with $Y_0\in\cT(T_0),X_0\in\cF(T_0)$ is a well-defined injective map.
\end{PropositionS}
\begin{proof}
First, we need to show that $\Phi$ is well defined. Let $T=Y_0\oplus X_0$  be a tilting $\Lambda$-module such that $Y_0\in\mathcal{T}(T_0)$, $X_0\in\mathcal{F}(T_0)$. Since $T_0$ is admissible, Theorem \ref{perfect} implies $\Phi(T)$ is a tilting $B_0$-module. Suppose $T_I=Y_0^{(I)}\oplus X_0^{(I)}, T_J=Y_0^{(J)}\oplus X_0^{(J)}$ are two tilting $\Lambda$-modules satisfying $\mathsf{add} (T_I)=\mathsf{add}(T_J)$. It suffices to verify $\mathsf{add}(\Phi(T_I))=\mathsf{add}(\Phi(T_J))$. By the definition of $\Phi$ and the assumptions on $T_I,T_j$, this is obvious.

Without loss of generality, let $T_1=Y_1\oplus X_1\in\mathcal{T}_{\mathcal{T}\cup\mathcal{F}}$, $T_2=Y_2\oplus X_2\in\mathcal{T}_{\mathcal{T}\cup\mathcal{F}}$ be two basic tilting $\Lambda$-modules with $X_1,X_2\in\mathcal{F}(T_0)$ and $Y_1,Y_2\in\mathcal{T}(T_0)$, such that $\Phi(T_1)=\Phi(T_2)$. Then by the definition of $\Phi$, we have $\Ext^{1}_\Lambda(T_0,X_1)=\Ext^{1}_\Lambda(T_0,X_2)$. While the Tilting Theorem \ref{Th0} implies $X_1=X_2$, so we let 
$X\coloneqq X_1=X_2$.
Consider the following two short exact sequences:
$$\xymatrix@1{
0\ar[r]&Tr_{Y_1}X\ar[r]&Y_1\ar[r]&Y_1/Tr_{Y_1}X\ar[r]&0&(b_1)}$$

$$\xymatrix@1{
0\ar[r]&Tr_{Y_2}X\ar[r]&Y_2\ar[r]&Y_2/Tr_{Y_2}X\ar[r]&0.&(b_2)}$$
Since $\Phi(T_1)=\Phi(T_2)$, the definition of $\Phi$ implies $\Hom_\Lambda(T_0,Y_1/Tr_{Y_1}X)=\Hom_\Lambda(T_0,Y_2/Tr_{Y_2}X)$. 
Since torsion class is closed under taking factor modules, we have $Y_1/Tr_{Y_1}X,Y_2/Tr_{Y_2}X\in\mathcal{T}(T_0)$. Then the Tilting Theorem \ref{Th0} forces $Y_1/Tr_{Y_1}X=Y_2/Tr_{Y_2}X$. For convenience, let $Y/Tr_{Y}X\coloneqq Y_1/Tr_{Y_1}X=Y_2/Tr_{Y_2}X$.
Since $Tr_{Y_i}X$ is generated by $X$ and $T_2$ is tilting, $\Ext^{1}_\Lambda(Y_2, Tr_{Y_1}X)=0$.
Since $Y_2$ is partial tilting, $(b_2)$ and the right exactness of $\Ext^{1}_\Lambda(Y_2,-)$ force $\Ext^{1}_\Lambda(Y_2,Y/Tr_{Y}X)=0$. Then applying the functor $\Hom_\Lambda(Y_2,-)$ to $(b_1)$, we obtain $\Ext^{1}_\Lambda(Y_2,Y_1)=0$. Similarly, we have $\Ext^{1}_\Lambda(Y_1,Y_2)=0$. By the Remark before this Proposition, we know $X\oplus Y_1\oplus Y_2$ is tilting. But since $X\oplus Y_1$ and $X\oplus Y_2$ are basic tilting modules, this implies $Y_1=Y_2$, thus $\Phi$ is injective.
\end{proof}

A tilting module $T_A$ is called \textbf{BB-tilting}, provided $T_A$ is a classic tilting module of the form $T=P[i]\oplus\tau^{-1}S$ for some simple but non-injective $S$, and it is called \textbf{APR-tilting}, if moreover, $S$ is projective. From now on, we suppose $T_0=P[i]\oplus \tau^{-1}S$ is a BB-tilting $\Lambda$-module and $B_0=\End_\Lambda (T_0)$. In that case, we say $B_0$ is a \textbf{BB-tilted algebra}.

Part of the statements in the following Lemma are probably known, but since we could not find a suitable reference, we give a proof here.

\begin{LemmaS}\label{Simple}
Suppose $T_0=P[i]\oplus \tau^{-1}S$ is BB-tilting, $B_0=\End_\Lambda (T_0)$. Then the following statements hold:
\begin{enumerate}
\item $T_0$ is admissible;
\item $\Ext^{1}_\Lambda(T_0,S)$ is the unique indecomposable modules in $\cX(T_0)$;
\item $B_0$ is hereditary if and only if $T_0$ is APR-tilting. Moreover, ${\rm gl.dim }B_0=2$ if and only if $T_0$ is not APR tilting. 
\end{enumerate}
\end{LemmaS}
\begin{proof}
\begin{enumerate}[fullwidth,itemindent=2em]
\item 
According to the definition of admissible tilting modules, to show that $T_0$ is admissible, it suffices to show no indecomposable summand of $M$ is generated by $S$, where $T=S\oplus M$ is a tilting $\Lambda$-module with $M\in\mathcal{T}(T_0)$. If there exists some indecomposable summand $M_i$ of $M$ that is generated by S, we obtain an exact sequence $0\to Z\to S^{r}\to M_i \to 0$. But since $S$ is simple and $\Ext^{1}_\Lambda(M,S)=0$, this yields a contradiction.
\item 
For any $X\in\mathcal{F}(T_0)$, $\Hom_\Lambda(P[i],X)=0$, so $S$ is the unique indecomposable $\Lambda$-module in $\cF(T_0)$.   Then according to Theorem \ref{Th0}, $\Ext^{1}_\Lambda(T_0,S)$ is the unique indecomposable module in $\cX(T_0)$. 
\item
If $T_0$ is APR-tilting, according to \cite[Proposition 1.13]{APR}, $B_0$ is hereditary.
If $T_0$ is not APR-tilting, we claim ${\rm pd}_{B_0}\Ext^{1}_\Lambda(T_0,S)$ $=2$. According to \cite[Theorem 5.2]{HR}, ${\rm gl.dim} B_0\leq 2$. Since all projective $B_0$-modules are in $\mathcal{Y}(T_0)$, $\Ext^{1}_\Lambda(T_0,S)$ is not projective. If ${\rm pd}_{B_0}\Ext^{1}_\Lambda(T_0,S)=1$, we consider its minimal projective resolution:
$$\xymatrix@1{0\ar[r]& \Hom_\Lambda(T_0,T^{'})\ar[r]& \Hom_\Lambda(T_0,T^{''})\ar[r]& \Ext^{1}_\Lambda(T_0,S)\ar[r]& 0}$$
with $T^{'}, T^{''}\in \mathsf{add} T_0$. Applying the functor $-\otimes T_0$ to it, \cite[1.6]{KB} yields a short exact sequence
$$\xymatrix@1{0\ar[r]& S\ar[r]& T^{'}\ar[r]& T^{''}\ar[r]& 0.}$$
Lemma \ref{L3} forces $\Hom_{B_0}(\Hom_\Lambda(T_0,P[i]),\Ext^{1}_{\Lambda}(T_0,S))=0$. Since $B=\Hom_\Lambda(T_0,P[i]\oplus \tau^{-1}S)$, we know  $\Ext^{1}_\Lambda(T_0,S)$ is simple, and hence $T^{''}=\tau^{-1}S$.  By the minimal projective resolution, we know $\tau^{-1}S$ cannot be a direct summand of $T^{'}$. Then $T^{'}\in\mathsf{add}P[i]$, and hence is projective. Since $\Lambda$ is hereditary, $S$ is projective, contradicts our assumption that $T_0$ is not APR-tilting. So we have ${\rm pd}_{B_0}(\Ext^{1}_\Lambda(T_0,S))=2$ and ${\rm gl.dim }B_0=2$. According to \cite[1.7]{KB}, ${\rm pd}_{B_0}Y\leq 1$ for all $Y\in\cY(T_0)$, and since $(\cX(T_0),\cY(T_0))$ is splitting, we know $S^{'}=\Ext^{1}_\Lambda(T_0,S)$ is the unique indecomposable $B_0$-module of projective dimension 2.
\end{enumerate}
\end{proof}

According to Proposition \ref{injection}, for every BB-tilted algebra $B_0=\End_\Lambda(T_0)$, we have an injective map $\Phi:\mathcal{T}_{\mathcal{T}\cup\mathcal{F}}\to \mathcal{T}_{B_0}$.

\begin{PropositionS}\label{surjective}
When $T_0$ is a BB-tilting module and $B_0=\End_\Lambda(T_0)$, the map $\Phi:\mathcal{T}_{\mathcal{T}\cup\mathcal{F}}\to \mathcal{T}_{B_0}$ is a  bijection.
\end{PropositionS}
\begin{proof}
Without loss of generality, we assume that $T^{'}\in\mathcal{T}_{B_0}$ is a basic tilting $B_0$-module. Recall that $(\mathcal{X}(T_0),\mathcal{Y}(T_0))$ is splitting in mod $B_0$.

If $T^{'}\in\mathcal{Y}(T_0)$, by Theorem \ref{Th0}, there exists $T\in\mathcal{T}(T_0)$ such that $T^{'}=\Hom_\Lambda(T_0,T)$. It's easy to see that $T$ is tilting and $\Phi(T)=T^{'}$.

If $T^{'}=S^{'}\oplus N^{'}$ with $N^{'}\in\mathcal{Y}(T_0)$, $S^{'}=\Ext^{1}_\Lambda(T_0,S)$. Theorem \ref{Th0} implies there exists $N\in\mathcal{T}(T_0)$ such that $N^{'}=\Hom_\Lambda(T_0,N)$. By Lemma \ref{L4}, $\Ext^{1}_\Lambda(S,N)\cong \Ext^{2}_{B_0}(S^{'},N^{'})=0$ (\ding{182}) and $\Hom_\Lambda(S,N)\cong \Ext^{1}_{B_0}(S^{'},N^{'})=0$ (\ding{183}). By Lemma \ref{L2}, $\Ext^{1}_\Lambda(S,S)\cong \Ext^{1}_{B_0}(S^{'},S^{'})=0$ (\ding{184}) and Lemma \ref{L1} implies $\Ext^{1}_\Lambda(N,N)\cong \Ext^{1}_{B_0}(N^{'},N^{'})=0$. So if $\Ext^{1}_\Lambda(N,S)=0$, $S\oplus N$ is exceptional and according to the Tilting Theorem \ref{Th0}, we know it's perfect exceptional and hence tilting.  Moreover, it's not difficult to see that $\Phi(S\oplus N)=T^{'}$. 
Since $T^{'}=S^{'}\oplus N^{'}$ is basic tilting, let $N=\bigoplus^{n-1}_{i=1}N_i$ be a decomposition of $N$ into indecomposable summands and set $\dim_k\Ext^{1}_\Lambda(N_i,S)=:r_i$. 
Consider the following exact sequence 
$$\xymatrix@1{0\ar[r]&S^{r_i}\ar[r]& M_i\ar[r]& N_i\ar[r]& 0&(c_i)}$$
such that the connecting homomorphism $\delta_i:\Hom_\Lambda(S^{r_i},S)\to \Ext^{1}_\Lambda(N_i,S)$ is surjective.
Taking the direct sum of $(c_i)$, we obtain
$$\xymatrix@1{0\ar[r]& S^{r}\ar[r]& M\ar[r] &N\ar[r] &0&(c)}$$
and $M=\bigoplus^{n-1}_{i=1}M_i$.
The construction of $(c)$ implies the connecting homomorphism $\delta:\Hom_\Lambda(S^{r},S)\to \Ext^{1}_\Lambda(N,S)$ is surjective.
In view of $\Hom_\Lambda(N,S)=0$ and \ding{184}, application of the functor $\Hom_\Lambda(-,S)$ to $(c)$ yields the following exact sequence
$$\xymatrix@1{0\ar[r]& \Hom_\Lambda(M,S)\ar[r]& \Hom_\Lambda(S^{r},S)\ar[r]^{\delta}& \Ext^{1}_\Lambda(N,S)\ar[r]& \Ext^{1}_\Lambda(M,S)\ar[r] &0.}$$
Since $\delta$ is surjective and $\dim_k\Hom_\Lambda(S^{r},S)$ $=\dim_k\Ext^{1}_\Lambda(N,S)=r$, it follows that $\delta$ is an isomorphism and hence $\Ext^{1}_\Lambda(M,S)=\Hom_\Lambda(M,S)=0$ (\ding{185}). Since $\mathsf{add}S=\mathcal{F}(T_0)$ , $\Hom_\Lambda(M,S)=0$ implies $M\in\mathcal{T}(T_0)$.
Applying the functor $\Hom_\Lambda(T_0,-)$ to $(c)$ yields a short exact sequence:
$$\xymatrix{0\ar[r]&\Hom_\Lambda(T_0,M)\ar[r]&\Hom_\Lambda(T_0,N)\ar[r]& \Ext^{1}_{\Lambda}(T_0,S^{r})\ar[r]& 0.&(d)}$$
Applying the functor $\Hom_{B_0}(-,\Hom_\Lambda(T_0,N))$ to $(d)$, we obtain an  exact sequence 
$$\xymatrix@1{\Ext^{1}_{B_0}(N^{'},N^{'})\ar[r]& \Ext^{1}_{B_0}(\Hom_\Lambda(T_0,M),\Hom_\Lambda(T_0,N))\ar[r]& \Ext^{2}_{B_0}({S^{'}}^{r},N^{'}).}$$
Since $S^{'}\oplus N^{'}$ is tilting, this implies $\Ext^{1}_{B_0}(\Hom_\Lambda(T_0,M),\Hom_\Lambda(T_0,N))=0$. By  Lemma \ref{L1}, we have $\Ext^{1}_\Lambda(M,N)\cong \Ext^{1}_{B_0}(\Hom_\Lambda(T_0,M),\Hom_\Lambda(T_0,N))=0$ (\ding{186}). In view of \ding{182} and \ding{184}, application of  the functor $\Hom_\Lambda(S,-)$ to $(c)$ gives $\Ext^{1}_\Lambda(S,M)=0$.
Applying the functor $\Hom_\Lambda(M,-)$ to $(c)$, while observing \ding{185} and \ding{186}, we obtain $\Ext^{1}_\Lambda(M,M)=0$. Then \ding{185} implies that $S\oplus M$ is exceptional.
In view of \ding{183} and \ding{185}, Lemma \ref{L5} implies that each $M_i$ is indecomposable.
For any isomorphism $\phi:M_i\to M_j$, consider the following commutative diagram with exact rows:
$$\xymatrix{0\ar[r]&S^{r_i}\ar[r]^{f_i}\ar@{-->}[d]^{\varphi}&M_i\ar[r]^{g_i}\ar[d]^{\phi}&N_i\ar[r]\ar@{-->}[d]^{\psi}&0\\
0\ar[r]&S^{r_j}\ar[r]^{f_j}&M_j\ar[r]^{g_j}&N_j\ar[r]&0,\\
}$$
where the existence of $\psi$ is ensured by \ding{183}.
By the Snake Lemma, $\varphi$ is injective, $\psi$ is surjective and $ker(\psi)\cong coker(\varphi)$. Since $coker(\varphi)\in \mathsf{add} S$, \ding{183} implies $\psi$ is an isomorphism, i.e., $N_i\cong N_j$ and $i=j$. So  $M$ is basic and $S\oplus M$ is perfect exceptional an hence tilting. 
By \ding{183}, $(c)$ is the canonical sequence associated to the torsion pair $(\mathsf{Gen} S,\mathcal{F}(S))$. By the uniqueness of the canonical sequence and the definition of $\Phi$, we have $\Phi(S\oplus M)=T^{'}$.
\end {proof}

Let $T_1,T_2\in\mathcal{T}_A$, and recall a partial order $\leq$ on $\mathcal{T}_A$ is defined via: $T_1\leq T_2$ if and only if $\mathsf{T_1^{\perp}}\subset \mathsf{T_2^{\perp}}$. Any subset of $\mathcal{T}_A$ inherits this partial order. The tilting quiver $\overrightarrow{\mathcal{K}_{A}}$ is defined as follows: the vertices are elements of $\mathcal{T}_A$, and for two basic tilting $A$-modules, there exists an arrow from $T_1$ to $T_2$ if and only if $T_1=M\oplus X$, $T_2=M\oplus Y$ with $X,Y$ being indecomposable summands belonging to a non-split exact sequence $0\to X\to\overline{M}\to Y\to 0$ with $\overline{M}\in \mathsf{add} M$. For a subset $\mathcal{D}\subset \mathcal{T}_A$, $\overrightarrow{\mathcal{D}}$ denotes the full subquiver of $\overrightarrow{\mathcal{K}_{A}}$ consisting of the elements in $\mathcal{D}$.

In \cite{SL}, the author considered the relationship between $(\mathcal{T}_{\Lambda},\leq)$ and $(\mathcal{T}_{B_0},\leq)$ when $T_0$ is an APR-tilting $\Lambda$-module and $B_0=\End_\Lambda(T_0)$. In this section, we will consider the relationship between $\overrightarrow{\mathcal{K}_{\Lambda}}$ and $\overrightarrow{\mathcal{K}_{B_0}}$ when $T_0$ is BB-tilting and $B_0=\End_\Lambda(T_0)$.

Given a partially ordered set $(\mathcal{C},\leq)$, we can define a diagram which is called the \textbf{Hasse diagram} whose vertices are elements of $\mathcal{C}$. There exists an arrow $C_1\to C_2$ with $C_1,C_2\in\mathcal{C}$ if and only if 
$C_1$ is a minimal element of $\{D|C_2\leq D\}$. The following Theorem states the relationship between the tilting quiver $\overrightarrow{\mathcal{K}_A}$ and the Hasse diagram of $(\mathcal{T_A}, \leq)$.

\begin{TheoremS}\cite[Theorem 2.1]{HaUn2}\label{Hasse}
$\overrightarrow{\mathcal{K}_A}$ is the Hasse diagram of $(\mathcal{T}_A, \leq)$.
\end{TheoremS}

\begin{LemmaS}\label{middle}
Let
$\xymatrix@1{0\ar[r]&X\ar[r]^-{\begin{psmallmatrix}
f_1\\
f_2
\end{psmallmatrix}}&Y\oplus Y^{'}\ar[r]^-{\begin{psmallmatrix}
g_1& g_2
\end{psmallmatrix}}&Z\ar[r]&0
}$
 be a non-split short exact sequence. If $Z$ is indecomposable and $\Hom_A(X,Y^{'})=0$, then $Y^{'}=0$.
\end{LemmaS}
\begin{proof}
Since $\Hom_A(X,Y^{'})=0$, we have $f_2=0$. Since $ker\begin{psmallmatrix}
g_1& g_2
\end{psmallmatrix}=Im\begin{psmallmatrix}
f_1\\
0
\end{psmallmatrix}\subset ker g_1\begin{psmallmatrix}
id_Y& 0
\end{psmallmatrix} $, 
there exists $h:Z\to Z$ such that the following diagram with exact rows commutes 
$$\xymatrix{
0\ar[r]&X\ar[r]^-{\begin{psmallmatrix}
f_1\\
0
\end{psmallmatrix}}\ar[d]_{id_X}&Y\oplus Y^{'}\ar[r]^-{\begin{psmallmatrix}
g_1& g_2
\end{psmallmatrix}}\ar[d]^-{\begin{psmallmatrix}
id_Y& 0
\end{psmallmatrix}}&Z\ar[r]\ar[d]^{h}&0\\
0\ar[r]&X\ar[r]^{f_1}&Y\ar[r]^{g_1}&Z.
}$$
Since $Z$ is indecomposable, $h$ is either nilpotent or invertible.
Since  $h\begin{psmallmatrix}
g_1 & g_2
\end{psmallmatrix}=g_1\begin{psmallmatrix}
id_Y & 0
\end{psmallmatrix}=
\begin{psmallmatrix}
g_1 & g_2
\end{psmallmatrix}\begin{psmallmatrix}
id_Y&0\\
0&0
\end{psmallmatrix}
$, $h^{i}\begin{psmallmatrix}
g_1 & g_2
\end{psmallmatrix}=\begin{psmallmatrix}
g_1 & g_2
\end{psmallmatrix}\begin{psmallmatrix}
id_Y&0\\
0&0
\end{psmallmatrix}^{i}=
\begin{psmallmatrix}
g_1 & 0
\end{psmallmatrix}
$ for $i\geq 1$.
If $h$ is nilpotent, then $g_1=0$ which implies $f_1$ is an isomorphism, so that the original sequence splits, a contradiction. If $h$ is invertible, then $h$ is injective and then the Snake Lemma implies $\begin{psmallmatrix}
id_Y&0
\end{psmallmatrix}$ is injective. Thus, $Y^{'}\subset ker\begin{psmallmatrix}
id_Y&0
\end{psmallmatrix}=(0)$, as desired.

\end{proof}

Recall that $\Lambda$ is hereditary and $T_0$ is a tilting $\Lambda$-module. As mentioned in the introduction $\mathcal{T}_{\mathcal{T}}\subset\mathcal{T}_\Lambda$ is the subset of $\mathcal{T}_\Lambda$ consisting of those equivalence classes whose representatives belong to $\cT(T_0)$, $\mathcal{T}_{\mathcal{F}}\subset\mathcal{T}_\Lambda$ is the subset of $\mathcal{T}_\Lambda$ consisting of those equivalence classes whose representatives belong to $\mathcal{F}(T_0)$.  
$\mathcal{T}_{\mathcal{T}\cup\mathcal{F}}$ is the subset of $\mathcal{T}_\Lambda$ consisting of those equivalence classes whose representatives have indecomposable direct summands only in $\mathcal{T}(T_0)\cup \mathcal{F}(T_0)$ 
and $\mathcal{T}_{\mathcal{T},\mathcal{F}}\subset\mathcal{T}_\Lambda$ is the subset of $\mathcal{T}_\Lambda$ consisting of elements in $\mathcal{T}_{\mathcal{T}\cup\mathcal{F}}$ but not in the union of $\mathcal{T}_{\mathcal{T}}$ and $\mathcal{T}_{\mathcal{F}}$. Similarly, we have $\mathcal{T}_{\mathcal{Y}},\mathcal{T}_{\mathcal{X}},\mathcal{T}_{\mathcal{X}\cup\mathcal{Y}}\subset\mathcal{T}_{B_0}$ and $\mathcal{T}_{\mathcal{X},\mathcal{Y}}\subset\mathcal{T}_{B_0}$.

Note that, when $T_0$ is BB-tilting, according to Lemma \ref{Simple}, $\mathcal{F}(T_0)=\mathsf{add}S$ and $\mathcal{X}(T_0)=\mathsf{add}(S^{'})$. Since $(\cX(T_0),\cY(T_0))$ is splitting (\cite[1.7]{KB}), $\mathcal{T}_{\mathcal{X},\mathcal{Y}}$ is just the subset of $\mathcal{T}_{B_0}$ consisting of those equivalence classes whose representatives contain $S^{'}$ as a summand. But  since $(\cT(T_0),\cF(T_0))$ may be not splitting, $\mathcal{T}_{\mathcal{T},\mathcal{F}}$ is contained in the subset of $\mathcal{T}_\Lambda$ consisting of those equivalence classes whose representatives contain $S$ as a summand.

For a quiver $\Gamma$, $\Gamma_0$ means the set of vertices of $\Gamma$.
A full subquiver $\Delta$ of a quiver $\Gamma$ is defined to be \textbf{convex} in $\Gamma$ if, for any path $x_0\to x_1\to\cdots\to x_t$ in $\Gamma$ with $x_0,x_t\in \Delta_0$, we have $x_i\in\Delta_0$ for all $i$ such that $0\leq i\leq t$.

\begin{PropositionS}\label{convex}
When $T_0$ is a BB-tilting module,
$\overrightarrow{\mathcal{T}_{\mathcal{T},\mathcal{F}}}$ is convex in $\overrightarrow{\mathcal{K}_\Lambda}$, 
$\overrightarrow{{\mathcal{T}_{\mathcal{X},\mathcal{Y}}}}$ is convex in $\overrightarrow{\mathcal{K}_{B_0}}$.
\end{PropositionS} 
\begin{proof}
For convenience, in this case we will write $\mathcal{T}_S$ and $\mathcal{T}_{S^{'}}$ instead of $\mathcal{T}_{\mathcal{T},\mathcal{F}}$ and $\mathcal{T}_{\mathcal{X},\mathcal{Y}}$, respectively.
Suppose there exists a path $T^{'}_1\to\cdots\to T^{'}\to\cdots\to T^{'}_2$ in $\overrightarrow{\mathcal{K}_{B_0}}$ such that $T_1^{'}, T_2^{'}\in\mathcal{T}_{S^{'}}$, i.e., $S^{'}$ is a summand of $T_1^{'}$ and $T^{'}_2$. Hence we can assume $T_1^{'},T_2^{'}$ are two basic tilting $B_0$-modules of the form: $T_1^{'}=S^{'}\oplus M^{'}_1,T_2^{'}=S^{'}\oplus M^{'}_2$ . By Theorem \ref{Hasse}, we know that $T^{'}_2\leq T^{'} \leq T^{'}_1$, i.e., $\mathsf{(S^{'}\oplus M_2^{'})^{\perp}}\subset \mathsf{(T^{'})^{\perp}}\subset \mathsf{(S^{'}\oplus M^{'}_1)^{\perp}}$, whence $\Ext^{i}_{B_0}(T^{'},S^{'})=\Ext^{i}_{B_0}(S^{'},T^{'})=0$ for $i\geq 1$. So $S^{'}\oplus T^{'}$ is exceptional, then the Remark before Proposition \ref{injection} implies $\mathsf{add}(T^{'})=\mathsf{add}(S^{'}\oplus T^{'})$. Consequently, $S^{'}$ is a direct summand of $T^{'}$, so $T^{'}\in\mathcal{T}_S^{'}$ and hence $\overrightarrow{\mathcal{T}_{S^{'}}}$ is convex in $\overrightarrow{\mathcal{K}_{B_0}}$.

Suppose there exists a path: $T_1\to\cdots\to T\to\cdots\to T_n$ in $\overrightarrow{\mathcal{K}_{\Lambda}}$ such that $T_1,T_n\in\mathcal{T}_S$. We will prove our statement by induction on the length $n-1$. If $n=2$, there is nothing to show. We assume our statement to be true for $n=m\geq 2$. Suppose there exists a path : $T_1\to T_2\cdots\to T\to\cdots\to T_{m+1}$ in $\overrightarrow{\mathcal{K}_{\Lambda}}$ with $T_1,T_{m+1}\in\mathcal{T}_S$. According to the definition of the tilting quiver, we may assume $T_1=M\oplus X$, $T_2=M\oplus Y$ are two basic tilting modules with $X,Y$ being indecomposable, and we have a non-split exact sequence 
$$\xymatrix@1{ 0\ar[r]& X\ar[r]^{f}&\overline{M}\ar[r]^{g}& Y\ar[r]& 0&(\star)}$$
such that $\overline{M}\in \mathsf{add} M$. Since $T_1,T_{m+1}\in\mathcal{T}_S$ and $S$ is the unique indecomposable module in $\mathcal{F}(T_0)$, $S$ is a summand of $T_1$ and $T_{m+1}$. By the arguments of the first part of the proof, we know $S$ is a direct summand of $T_2$. Again, since $T_1\in\mathcal{T}_S$ and $T_1$ is basic, either $X=S$ or $S$ is a summand of $M$. If $X=S$, the existence of the non-split sequence $(\star)$ implies $\Ext^{1}_\Lambda(Y,S)\ne0$. But since $S$ is a direct summand of $T_2$ and $T_2=M\oplus Y$, the self-orthogonal property of tilting modules implies $\Ext^{1}_\Lambda(Y,S)=0$, a contradition. Consequently, $S$ is a direct summand of $M$. By the definition of $\mathcal{T}_S$ and the fact that $S$ is the unique indecomposable module in $\mathcal{F}(T_0)$, $T_1\in\mathcal{T}_S$ implies $S$ is the unique indecomposable summand of $T_1$ that does not belong to $\mathcal{T}(T_0)$. Hence $X\in\mathcal{T}(T_0)$, so that $S\in\mathcal{F}(T_0)$ implies $\Hom_\Lambda(X,S)=0$. Since $Y$ is indecomposable, Lemma \ref{middle} implies $S$ is not a summand of $\overline{M}$. But $T_1=M\oplus X\in\mathcal{T}_S$ and $S$ is the unique indecomposable module in $\mathcal{F}(T_0)$, so $\overline{M}\in\cT(T_0)$. Since $\mathcal{T}(T_0)$ is closed under taking factors modules, we obtain $Y\in\cT(T_0)$. From the definition of $\mathcal{T}_S$, we now obtain $T_2\in\mathcal{T}_S$ and the induction hypothesis shows that the path: $T_2\cdots\to T\to\cdots\to T_{m+1}$ belongs to $\overrightarrow{\mathcal{T}_S}$. Consequently, we have proved our statement.
\end{proof}

Note that, when $T\in\mathcal{T}(T_0)$ is a tilting $\Lambda$-module, the map $\Phi$ maps $T$ to $\Hom_\Lambda(T_0,T)$. Suppose that $T^{'}\in\mathcal{Y}(T_0)$ is a tilting $B_0$-module. It's easy to check that $T^{'}\otimes_{B_0} T_0$ is a tilting $\Lambda$-module such that $\Phi(T^{'}\otimes_{B_0}T_0)=T^{'}$. Then by Propositions \ref{injection}, \ref{Simple}, \ref{surjective}, when $T_0$ is BB-tilting, $\Phi_0=\Phi |_{\mathcal{T}_{\mathcal{T}}}:\mathcal{T}_{\mathcal{T}}\to \mathcal{T}_{\mathcal{Y}}$ and $\Phi_2=\Phi |_{\mathcal{T}_{\mathcal{T},\mathcal{F}}}:\mathcal{T}_{\mathcal{T},\mathcal{F}}\to \mathcal{T}_{\mathcal{X},\mathcal{Y}}$ are bijections. Obviously, $\Phi_0$ is a bijection of partially ordered sets and by \cite[Theorem2.7]{HaUn1}, we know $\Phi_2$ is also a bijection of partially ordered sets.

We are now in a position to prove Theorem \ref{The 2}.

\begin{TheoremS}\label{quiveriso}
When $T_0$ is a BB-tilting module, $\Phi_0$, $\Phi_2$  induce isomorphisms of quivers between $\overrightarrow{\mathcal{T}_{\mathcal{T}}}$ and $\overrightarrow{\mathcal{T}_{\mathcal{Y}}}$, $\overrightarrow{\mathcal{T}_{\mathcal{T},\mathcal{F}}}$ and $\overrightarrow{\mathcal{T}_{\mathcal{X},\mathcal{Y}}}$, respectively.
\end{TheoremS}
\begin{proof}
The isomorphism between $\overrightarrow{\mathcal{T}_{\mathcal{T}}}$ and $\overrightarrow{\mathcal{T}_{\mathcal{Y}}}$ is obvious.\\
As in the proof of Proposition \ref{convex}, we will write $\overrightarrow{\mathcal{T}_S}$ and $\overrightarrow{\mathcal{T}_{S^{'}}}$ instead of $\overrightarrow{\mathcal{T}_{\mathcal{T},\mathcal{F}}}$ and $\overrightarrow{\mathcal{T}_{\mathcal{X},\mathcal{Y}}}$, respectively.
Suppose there exists an arrow $T^{'}_1\to T^{'}_2$ in $\overrightarrow{\mathcal{T}_{{S}^{'}}}$ with $T^{'}_1,T^{'}_2$ being basic tilting. Since $T^{'}_1,T^{'}_2\in\mathcal{T}_{S^{'}}$, by the definition of the tilting quiver and $\mathcal{T}_{S^{'}}$, we assume $T_1^{'}=S^{'}\oplus C^{'}\oplus C_1^{'}$ and $T_2^{'}=S^{'}\oplus C^{'}\oplus C_2^{'}$ with $C^{'}_1,C_2^{'}$ being indecomposable.  According to Proposition \ref{injection} and Proposition  \ref{surjective}, $\Phi_2$ is bijective, so we have  $T_1\coloneqq\Phi_2^{-1}(T^{'}_1), T_2\coloneqq\Phi_2^{-1}(T^{'}_2)$. Moreover, according to the proof of Proposition \ref{surjective} (exact sequences $(c_i)$, $(c)$ there), we may assume $T_1=S\oplus X\oplus X_1$ and $T_2=S\oplus X\oplus X_2$ with $X_1,X_2$ being indecomposable. Since $T^{'}_2\leq T^{'}_1$ and $\Phi_2$ is a bijection of partial orders, we have $T_{2}\leq T_{1}$ i.e., $\mathsf{(S\oplus X\oplus X_2)^{\perp}}\subset \mathsf{(S\oplus X\oplus X_1)^{\perp}}$. Consequently, $\Ext^{1}_\Lambda(X_1,X_2)=0$, and we claim that $\Ext^{1}_\Lambda(X_2,X_1)\ne 0$. Otherwise, the Remark before Proposition \ref{injection} implies  $\mathsf{add} T_1=\mathsf{add}(T_1\oplus X_2)=\mathsf{add} T_2$, which contradicts $\Phi_2$ being injective. Thus $\Ext^{1}_\Lambda(X_2,X_1)\ne 0$ and \cite[Theorem1.1]{HaUn4} implies there exists an arrow $T_1\to T_2$.

Suppose there exists an arrow $T_1\to T_2$ in  $\overrightarrow{\mathcal{T}_{S}}$, let $T_1^{'}=\Phi_2(T_1)$, $T_2^{'}=\Phi_2(T_2)$. By Theorem \ref{Hasse}, $T_2\leq T_1$, and  since $\Phi_2$ is a bijection of partially ordered sets, $T^{'}_2\leq T^{'}_1$. If there exists $T^{'}\in\mathcal{T}_{B_0}$ such that  $T^{'}_2\leq T^{'}\leq T^{'}_1$, then by the proof of Proposition \ref{convex}, $T^{'}\in\mathcal{T}_{S^{'}}$. Since $\Phi_2$ is a bijection, there exists $T\in\mathcal{T}_S$ such that $T^{'}=\Phi_2(T)$. Since $\Phi_2$ is a bijection of partially  ordered sets, $T_2\leq T\leq T_1$. According to Theorem \ref{Hasse}, this contradicts the minimality  of $T_1$. So by Theorem \ref{Hasse}, there exists an arrow $T^{'}_1\to T_2^{'}$.
\end{proof}

Let $B_0=\End_\Lambda(T_0)$ be a BB-tilted algebra. Since the torsion pair $(\cX(T_0),\cY(T_0))$ splits and $\mathsf{add}(S^{'})=\cX(T_0)$, $\cT_{B_0}$ is the disjoint union of $\mathcal{T}_{\mathcal{Y}}$ and ${\mathcal{T}_{\mathcal{X},\mathcal{Y}}}$. According to Theorem \ref{quiveriso}, we can obtain the structure of $\overrightarrow{\mathcal{K}_{B_0}}$ from that of $\overrightarrow{\mathcal{K}_\Lambda}$ except for the arrows between elements of $\mathcal{T}_{\mathcal{Y}}$ and $\mathcal{T}_{\mathcal{X},\mathcal{Y}}$.

\begin{PropositionS}\label{arrow}
Suppose $T_0$ is a BB-tilting $\Lambda$-module. Then
\begin{enumerate}
\item 
There are no arrows from $\cT_{\mathcal{X},\mathcal{Y}}$ to $\cT_{\cY}$ in $\overrightarrow{\cK_{B_0}}$.
\item 
There are no arrows from $\cT_{\cT}$ to $\cT_{\mathcal{T},\mathcal{F}}$ in $\overrightarrow{\cK_{\Lambda}}$.
\end{enumerate}
\end{PropositionS}
\begin{proof}
\begin{enumerate}[fullwidth,itemindent=2em]
\item 
Let $T_1^{'}\in\mathcal{T}_{\mathcal{Y}}$, $T_2^{'}\in\cT_{\mathcal{X},\mathcal{Y}}$ be two basic tilting $B_0$-modules. Suppose there exists an arrow from $T_2^{'}$ to $T_1^{'}$,
then we can assume $T_2^{'}=N^{'}\oplus S^{'}$, $T^{'}_1=N^{'}\oplus Y^{'}\in\mathcal{Y}(T_0)$. By the definition of the tilting quiver, there exists a non-split exact sequence $0\to S^{'}\to \overline{N^{'}}\to Y^{'}\to 0$ with $\overline{N^{'}}\in \mathsf{add} N^{'}$. If $T_0$ is APR-tilting, then according to \cite[Proposition 1.13]{APR}, $S^{'}$ is injective. Hence the short exact sequence splits, a contradiction. Hence $T_0$ is not APR-tilting, and Lemma \ref{Simple} implies ${\rm pd}_{B_0}S^{'}=2$. But since $Y^{'},\overline{N^{'}}\in\mathcal{Y}(T_0)$, \cite[1.7]{KB} yields ${\rm pd}_{B_0} Y^{'}, {\rm pd}_{B_0}\overline{N^{'}}\leq 1$. Then applying the functor $\Hom_{B_0}(-,Z^{'})$ to the exact sequence for any $B_0$-module $Z^{'}$, we have $\Ext^{2}_{B_0}(S^{'},Z^{'})=0$ for any $B_0$-module $Z^{'}$, which implies ${\rm pd}_{B_0}S^{'}\leq 1$, a contradiction. 
\item
Let $T_1\in\mathcal{T}_{\mathcal{T}}$, $T_2\in\cT_{\mathcal{T},\mathcal{F}}$ be two basic tilting $\Lambda$-modules and suppose there exists an arrow from $T_1$ to $T_2$, then we can assume $T_2=M\oplus S$, $T_1=M\oplus Y\in\mathcal{T}(T_0)$. If there exists an arrow from $T_1$ to $T_2$, then  there exists a non-split exact sequence $0\to Y\to \overline{M}\to S\to 0$ with $\overline{M}\in \mathsf{add} M$. But since $M\in\mathcal{T}(T_0)$, $S\in\mathcal{F}(T_0)$, we have $\Hom_\Lambda(M,S)=0$. Hence $\Hom_\Lambda(\overline{M},S)=0$, a contradiction.
\end{enumerate}
\end{proof}

Recall that a partial tilting $A$-module $M$ is called \textbf{almost complete tilting}, if $M$ admits $n-1$ non-isomorphic indecomposable summands. A basic $A$-module $Z$ is called a \textbf{complement} to $M$, if $M\oplus Z$ is tilting and $\mathsf{add} M\cap\mathsf{add} Z=0$.

\begin{PropositionS}\label{arrow2}
Suppose $T_0$ is a BB-tilting $\Lambda$-modules, $T_1^{'}\in\mathcal{T}_{\cY}$, $T_2^{'}\in\cT_{\cX,\cY}$ are two basic tilting modules, such that there is an arrow 
$T^{'}_1\to T^{'}_2$ in $\overrightarrow{\cK_{B_0}}$. Then
\begin{enumerate}[fullwidth,itemindent=2em]
\item 
Let $T_i=\Phi^{-1}(T_i^{'})$ for $i=1,2$. Then there exists $E_1,E_2,Y\in\cT(T_0)$ such that $T_1=E_1\oplus Y$, $T_2=E_2\oplus S$.
\item 
There exists an arrow $T_2\to T_1$ in $\overrightarrow {\cK_{\Lambda}}$ if and only if $\Ext^{1}_\Lambda(E_1,E_2)=0$.
\end{enumerate}
\end{PropositionS}
\begin{proof}
\begin{enumerate}[fullwidth,itemindent=2em]
\item 
Since  $T_1^{'}\in\mathcal{T}_{\cY}$, $T_2^{'}\in\cT_{\cX,\cY}$ are two basic tilting modules, by the definition of $\cT_{\cY}$ and $\cT_{\cX,\cY}$ and the assumption on the existence of the arrow $T^{'}_1\to T^{'}_2$, we may assume $T_1^{'}=E_1^{'}\oplus Y^{'}$ and $T_2^{'}=S^{'}\oplus E_1^{'}$ such that $E_1^{'},Y^{'}\in\cY(T_0)$. Let $E_1\coloneqq E_1^{'}\otimes_{B_0}T_0$, $Y\coloneqq Y^{'}\otimes_{B_0}T_0$, then $T_1=E_1\oplus Y$ satisfies $\Phi(T_1)=T_1^{'}$.
In analogy with the construction of $(c)$ in the proof of Proposition \ref{surjective},
consider the exact sequence 
$$\xymatrix@1{0\ar[r]&S^{r}\ar[r]&E_2\ar[r]&E_1\ar[r]&0.&(\star)}$$
Where $E_1\coloneqq E_1^{'}\otimes_{B_0}T_0$ and $r\coloneqq\dim_k\Ext^{1}_\Lambda(E_1,S)$, such that the connecting hommorphism $\Hom_\Lambda(S^{r},S)\to \Ext^{1}_\Lambda(F,S)$ is surjective. Then the proof of Proposition \ref{surjective} shows $T_2=E_2\oplus S$ is as desired.
\item 
By the definition of the tilting quiver, we have a non-split short exact sequence $0\to Y^{'}\to \overline{E_1^{'}}\to S^{'}\to 0$ such that $\overline{E_1^{'}}\in \mathsf{add} E_1^{'}$. Applying the functor $-\otimes_{B_0}T_0$ to the short exact sequence, \cite[1.6]{KB} yields a short exact sequence 
$$\xymatrix@1{0\ar[r]& S\ar[r]& Y\ar[r]& \overline{E_1}\ar[r]& 0&(\star\star)}$$
such that $\overline{E_1}\in \mathsf{add} E_1$. 
Since $E_2\oplus S$ is tilting, $\Ext^{1}_\Lambda(E_2,E_2)=0$.
Applying the functor $\Hom_\Lambda(E_2,-)$ to $(\star)$, 
we have $\Ext^{1}_\Lambda(E_2,E_1)=0$(\ding{172}). By the same token, applying the functor $\Hom_\Lambda(E_2,-)$ to $(\star\star)$, we have $\Ext^{1}_\Lambda(E_2,Y)=0$(\ding{173}). 

If $\Ext^{1}_\Lambda(E_1,E_2)=0$(\ding{174}), $\Ext^{1}_\Lambda(\overline{E_1},E_2)=0$. Since $E_2\oplus S$ is tilting, $\Ext^{1}_\Lambda(S,E_2)=0$, so that applying the functor $\Hom_\Lambda(-,E_2)$ to $(\star\star)$, we have
$\Ext^{1}_\Lambda(Y,E_2)=0$(\ding{175}). Since $E_1\oplus Y$ is tilting, \ding{172},\ding{173},\ding{174},\ding{175} imply that $E_2\oplus E_1\oplus X$ is self-orthogonal. Then the Remark before Proposition \ref{injection} implies $E_2\in \mathsf{add}(E_1\oplus Y)$. This means $E_2$ is a common summand of the two basic tilting modules  $E_1\oplus Y$ and $E_2\oplus S$ . Therefore, \cite[Theorem 1.1]{HaUn4} provides an arrow between $T_1$ and $T_2$, and by Proposition \ref{arrow}, we know there exists an arrow from $T_2$ to $T_1$.

If there exists an arrow from $T_2$ to $T_1$, then by the definition of the tilting quiver, $T_1$, $T_2$ possess a common almost complete tilting module as summand. Since $S\not\in \mathsf{add}(E_1\oplus Y)$, this means $E_2$ is the common almost complete tilting module, hence $E_2\in \mathsf{add} (E_1\oplus Y)$. Since $E_1\oplus Y$ is tilting, hence is exceptional, we have $\Ext^{1}_\Lambda(E_1,E_2)=0$.
\end{enumerate}
\end{proof}

\begin{Remark} According to \cite[Corollary 1.2]{HaUn4}, when $A$ is hereditary, every almost complete tilting module $M$ admits at most two complements. But this is not true for algebras of global dimension not smaller than 2.
\end{Remark}

\begin{PropositionS}\label{complement}
Suppose that $T_0$ is a BB-tilting $\Lambda$-module, $B_0=\End_\Lambda(T_0)$. If $M^{'}\in\mathcal{Y}(T_0)$ is a basic almost complete tilting $B_0$-module, then $M^{'}$ admits at most three complements.
\end{PropositionS}
\begin{proof}
Suppose $M^{'}\oplus Z^{'}$ is a basic tilting $B_0$-module. If $Z^{'}\in\mathcal{Y}(T_0)$,  let $M\oplus Z\in\mathcal{T}(T_0)$ such that $M^{'}\oplus Z^{'}= \Hom_\Lambda(T_0,M\oplus Z)$.  Then according to the Tilting Theorem \ref{Th0} and Lemma \ref{L1}, we know  that $M\oplus Z$ is perfect exceptional if and only if $M^{'}\oplus Z^{'}$ is perfect exceptional. 
Since ${\rm pd}_\Lambda(M\oplus Z)\leq 1$ and $M\oplus Z\in\mathcal{T}(T_0)$, \cite[1.7]{KB} implies ${\rm pd}_{B_0}(M^{'}\oplus Z^{'})\leq 1$. By Proposition \ref{perfecttilting}, perfect exceptional $B_0$ modules are tilting (and vice versa).  
Since $\Lambda$ is hereditary, as we mentioned above, $M$ admits at most two complements belong to $\mathcal{T}(T_0)$. Then according to the Tilting Theorem \ref{Th0},  $M^{'}$ admits at most two complements belong to $\mathcal{Y}(T_0)$. Since $(\mathcal{X}(T_0),\mathcal{Y}(T_0))$ is splitting in mod $B_0$ and $S^{'}$ is the unique indecomposable module in $\mathcal{X}(T_0)$, $M^{'}$ admits at most three complements.
\end{proof}

\begin{Remark} Let $T_0\in \text{mod }\Lambda$ be BB-tilting, $M^{'}\in\mathcal{Y}(T_0)$ be an almost complete tilting $B_0$-module, such that $M^{'}\oplus Y^{'}\in\mathcal{T}_{\mathcal{Y}}$, $M^{'}\oplus S^{'}\in\mathcal{T}_{\cX,\cY}$ are two basic tilting $B_0$-modules. According to \cite[Lemma 1.1]{HaUn1}, there either exists an arrow $M^{'}\oplus Y^{'}\to M^{'}\oplus S^{'}$ or there is a directed path $M^{'}\oplus Y^{'}\to M^{'}\oplus X^{'}\to M^{'}\oplus S^{'}$ in $\overrightarrow{\mathcal{K}_{B_0}}$.
\end{Remark}

\begin{PropositionS}\label{arrow3}
Suppose $T_0$ is a BB-tilting $\Lambda$-modules, $T_1\in\mathcal{T}_{\cT}$, $T_2\in\cT_{\cT,\cF}$ are two basic tilting modules, such that there is an arrow 
$T_2\to T_1$ in $\overrightarrow{\cK_{\Lambda}}$. Then
\begin{enumerate}[fullwidth,itemindent=2em]
\item 
Let $T_i^{'}=\Phi(T_i)$ for $i=1,2$. Then there exists $M^{'},N^{'},Y^{'}\in\cY(T_0)$ such that $T_1^{'}=M^{'}\oplus Y^{'}$, $T_2=N^{'}\oplus S^{'}$.
\item 
There exists an arrow $T_1^{'}\to T_2^{'}$ in $\overrightarrow {\cK_{B_0}}$ if and only if $\Ext^{1}_{\Lambda}(N,M)=0$, where $N,M\in\cT(T_0)$ such that $\Hom_\Lambda(T_0,N)=N^{'}$, $\Hom_\Lambda(T_0,M)=M^{'}$.
\end{enumerate}
\end{PropositionS}
\begin{proof}
\begin{enumerate}[fullwidth,itemindent=2em]
\item 
Since $T_1\in\mathcal{T}_{\cT}$, $T_2\in\cT_{\cT,\cF}$ are two basic tilting modules, and there exists an arrow $T_2\to T_1$ in $\overrightarrow{\cK_{\Lambda}}$, we assume 
$T_1=M\oplus Y$, $T_2=M\oplus S$ such that $M,Y\in\cT(T_0)$. Let $M^{'}\coloneqq \Hom_\Lambda(T_0,M)$, $Y\coloneqq \Hom_\Lambda(T_0,Y)$, then $T_1^{'}=M^{'}\oplus Y^{'}$ satisfies $T^{'}_1=\Phi(T_1)$. By the definition of $\Phi$, $T_2^{'}=\Phi(T_2)=\Hom_\Lambda(T_0,N)\oplus S^{'}$. Where $N$ is in the following exact sequence
$$\xymatrix@1{0\ar[r]&S^{r}\ar[r]&M\ar[r]&N\ar[r]&0&(\star)}$$
such that $\dim_k\Hom_\Lambda(S,M)=r$ and $\Hom_\Lambda(S,S^{r})\to \Hom_\Lambda(S,M)$ is surjective. Let $N^{'}\coloneqq \Hom_\Lambda(T_0,N)$, our statements follow.
\item
The existence of the arrow $T_2\to T_1$ yields a non-split short exact sequence  
$$\xymatrix{0\ar[r] &S\ar[r] &\overline{M}\ar[r]& Y\ar[r]& 0&(\star\star)}$$
with $\overline{M}\in \mathsf{add} M$. 
Since $M\oplus Y$ is tilting, $\Ext^{1}_\Lambda(M,M)=\Ext^{1}_\Lambda(Y,M)=0$. Applying the functor $\Hom_\Lambda(M,-)$ to $(\star)$, we have $\Ext^{1}_\Lambda(M,N)=0$. Applying the functor $\Hom_\Lambda(Y,-)$ to $(\star)$, we have $\Ext^{1}_\Lambda(Y,N)=0$. Then according to Lemma \ref{L1}, we have $\Ext^{1}_{B_0}(M^{'},N^{'})=\Ext^{1}_{B_0}(Y^{'},N^{'})$ $=0$(\ding{182}).

If $\Ext^{1}_\Lambda(N,M)=0$, then $\Ext^{1}_\Lambda(N,\overline{M})=0$. Applying the functor $\Hom_\Lambda(N,-)$ to $(\star\star)$, we have $\Ext^{1}_\Lambda(N,Y)=0$. According to Lemma \ref{L1}, then we have $\Ext^{1}_{B_0}(N^{'},M^{'})=\Ext^{1}_{B_0}(N^{'},Y^{'})=0$(\ding{183}). Since $M^{'}\oplus Y^{'}$ is tilting, $N^{'}$ is partial tilting and ${\rm pd}_{B_0}N^{'},{\rm pd}_{B_0}(M^{'}\oplus Y^{'})\leq 1$, \ding{182}, \ding{183} imply that $N^{'}\oplus M^{'}\oplus Y^{'}$ is self-orthogonal. Then the Remark before Proposition \ref{injection} implies $N^{'}\in \mathsf{add}(M^{'}\oplus Y^{'})$. But by our assumptions, $N^{'}$ and $M^{'}\oplus Y^{'}$ are basic modules with $n-1$ and $n$ indecomposable summands, respectively. So $N^{'}$ is a summand of $M^{'}\oplus Y^{'}$ and either $N^{'}\cong M^{'}$ or $Y^{'}$ is a summand of $N^{'}$. 
If $N^{'}\cong M^{'}$, then $N\cong M$ and $(\star)$ implies $\Hom_\Lambda(S,M)=0$, a contradiction. So $Y^{'}$ is a summand of $N^{'}$, and we can write $T_1^{'}=M^{'}\oplus Y^{'}=N^{'}\oplus X^{'}$ with $X^{'}\in \mathsf{add} M^{'}$(\ding{184}).

If there exists no arrow from $T_1^{'}$ to $T_2^{'}$, then according to the Remark after Proposition \ref{complement},  we can find a basic tilting $B_0$-module $T^{'}=N^{'}\oplus Z^{'}$ such that $T_1^{'}\to T^{'}\to T^{'}_2$ is a full subquiver of $\overrightarrow{\mathcal{K}_{B_0}}$. According to the definition of the tilting quiver, there exist two non-split short exact sequences $$\xymatrix@1{0\ar[r] &X^{'}\ar[r]& \overline{N^{'}}_1\ar[r]& Z^{'}\ar[r]& 0&(\star\star\star)}$$ 
$$\xymatrix@1{0\ar[r] &Z^{'}\ar[r] &\overline{N^{'}}_2\ar[r] &S^{'}\ar[r] &0&(\star\star\star\star)}$$ such that $\overline{N^{'}}_1,\overline{N^{'}}_2\in \mathsf{add} N^{'}$. Since $M\oplus S$ is tilting, $\Ext^{1}_\Lambda(S,M)=0$, and Lemma \ref{L4} implies $\Ext^{2}_{B_0}(S^{'},M^{'})=0$. Then \ding{184} implies $\Ext^{2}_{B_0}(S^{'},X^{'})=0$. Since $S^{'}\oplus N^{'}$ is tilting, $\Ext^{1}_{B_0}(S^{'},N^{'})=0$. Applying the functor $\Hom_{B_0}(S^{'},-)$ to $(\star\star\star)$, we have $\Ext^{1}_{B_0}(S^{'},Z^{'})=0$ which contradicts $(\star\star\star\star)$ being non-split. So there exists an arrow from $T_1^{'}$ to $T_2^{'}$.

Suppose there exists an arrow from $T_1^{'}$ to $T_2^{'}$. Since $T_1^{'}\in\mathcal{T}_{\mathcal{Y }}$, $N^{'}$ is the common almost complete tilting module of $T_1^{'}=N^{'}\oplus S^{'}$ and $T_1^{'}=M^{'}\oplus Y^{'}$. Since $T_1^{'}$ is self-orthogonal, $\Ext^{1}_{B_0}(N^{'},M^{'})=0$. Then by Lemma \ref{L1} we have $\Ext^{1}_\Lambda(N,M)=0$, so that our conclusion follows.
\end{enumerate}
\end{proof}

\section {examples}
\begin{example}
Let $\Lambda=kQ$ be the path algebra of the quiver $\xymatrix{1\ar[r]&2\ar[r]&3\ar[r]&4}$. Since it is of finite representation type, according to \cite[Corollary 2.2]{HaUn2},  $\overrightarrow{\mathcal{K}_\Lambda}$ is connected and finite. It's easy to see that $\Lambda_\Lambda$ is the unique source point and $D(_\Lambda\Lambda)$ is the unique sink point in $\overrightarrow{\mathcal{K}_\Lambda}$. Start with taking almost complete tilting modules from $\Lambda_\Lambda$ then consider its complements, and by the same token until reaching $D(_\Lambda\Lambda)$, we see that the tilting quiver $\overrightarrow{\mathcal{K}_\Lambda}$ is
$$
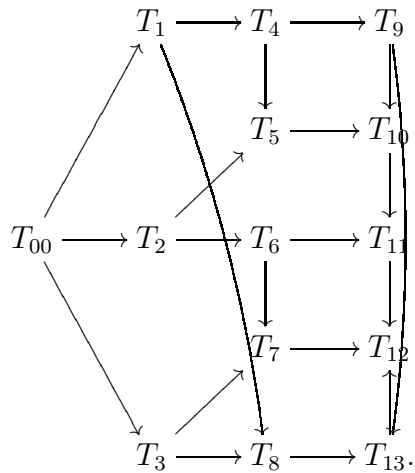
\xymatrix{&{T_1}\ar@/^/[rdddd]\ar[r]&{T_4}\ar[r]\ar[d]&{T_9}\ar[d]\ar@/^/[dddd]\\
&&{T_5}\ar[r]&{T_{10}}\ar[d]\\
{T_{00}}\ar[uur]\ar[ddr]\ar[r]&{T_2}\ar[ur]\ar[r]&{T_6}\ar[d]\ar[r]&{T_{11}}\ar[d]\\
&&{T_7}\ar[r]&{T_{12}}\\
&{T_3}\ar[r]\ar[ur]&{T_8}\ar[r]&{T_{13}}.\ar[u]&
}
$$

If we take $T_0=P[2]\oplus \tau^{-1}S(2)$ which is BB-tilting but not APR-tilting, then  $\mathcal{T}_{\mathcal{T}}=\{T_3,T_7,T_8,$ $T_{12},T_{13}\}$ and $\mathcal{T}_{S}=\{T_6,T_{11}\}$ are marked as below by ${\bullet}$ and $\star$, respectively, and the arrows between them are marked by $\dashrightarrow$:
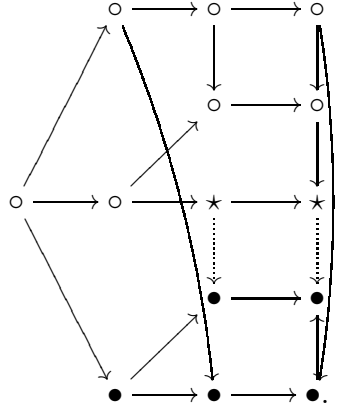
$$
\xymatrix{&{\circ}\ar@/^/[rdddd]\ar[r]&{\circ}\ar[r]\ar[d]&{\circ}\ar[d]\ar@/^/[dddd]\\
&&{\circ}\ar[r]&{\circ}\ar[d]\\
{\circ}\ar[uur]\ar[ddr]\ar[r]&{\circ}\ar[ur]\ar[r]&{\star}\ar@{.>}[d]\ar[r]&{\star}\ar@{.>}[d]\\
&&{\bullet}\ar[r]&{\bullet}\\
&{\bullet}\ar[r]\ar[ur]&{\bullet}\ar[r]&{\bullet}.\ar[u]&
}
$$

According to Theorem \ref{quiveriso}, we know $\overrightarrow{\mathcal{K}_{B_0}}$ is the union of 

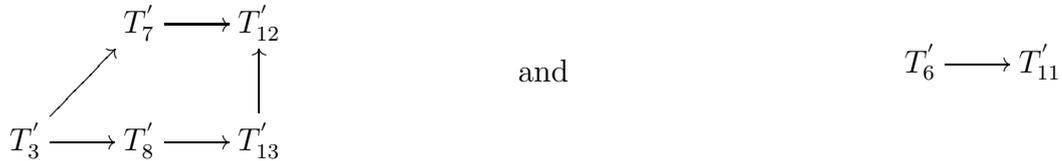
\begin{minipage}[t]{0.3\textwidth}

$$\xymatrix{
                 & {T_7^{'}}\ar[r] & {T_{12}^{'}}\\
     {T_3^{'}}\ar[r]\ar[ru] & {T_8^{'}}\ar[r] & {T^{'}_{13}}\ar[u]
}$$
\end{minipage}
\begin{minipage}[t]{0.3\textwidth}
\begin{center}
\vspace{30pt}
and 
\end{center}
\end{minipage}
\begin{minipage}[t]{0.3\textwidth}
\begin{center}
\vspace{15pt}
$$\xymatrix{&  {T^{'}_6}\ar[r]  &  {T_{11}^{'}}\\           
}$$
\end{center}
\end{minipage}\\
and some more arrows between them.

Since there exist arrows from $T_6$ to $T_7$ and $T_{11}$ to $T_{12}$ in $\overrightarrow{\mathcal{K}_\Lambda}$, by verifying the condition in Proposition \ref{arrow3} (b), we know there exist arrows from $T_7^{'}$ to $T_6^{'}$ and $T_{12}^{'}$ to $T_{11}^{'}$ in $\overrightarrow{\mathcal{K}_{B_0}}$. By the Remark after Proposition \ref{complement}, we know $T^{'}_6$ $T^{'}_{11}$ have only one predecessor in $\mathcal{T}_{\mathcal{Y}}$,  so the tilting quiver of $B_0$ is :
$$
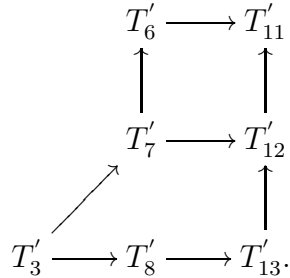
\xymatrix{&  {T^{'}_6}\ar[r]  &  {T_{11}^{'}}\\
                  & {T_7^{'}}\ar[r]\ar[u] & {T_{12}^{'}}\ar[u]\\
     {T_3^{'}}\ar[r]\ar[ru] & {T_8^{'}}\ar[r] & {T^{'}_{13}}.\ar[u]
}$$

By computations, $B_0=\End_\Lambda(T_0)$ is given by the path algebra of the following  quiver bounded by $\alpha\beta=0$,
$$
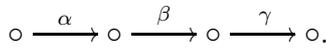
\xymatrix{{\circ}\ar[r]^{\alpha}&{\circ}\ar[r]^{\beta}&{\circ}\ar[r]^{\gamma}&{\circ}.}$$ 

If we take $T_0=P(1)\oplus P(4)\oplus I(1)\oplus I(2)$ which is not BB-tilting, then
 $\mathcal{T}_{\mathcal{T}}=\{T_7,T_{12}\}$ $\mathcal{T}_{\mathcal{T},\mathcal{F}}=\{T_6,T_9,T_{10},$ $T_{11},T_{13}\}$ and we can check that it is also admissible. By computations we know $B_0=\End_\Lambda(T_0)$ is given by the path algebra of the quiver  $$\xymatrix{{\circ}\ar[r]^{\alpha}&{\circ}\ar[r]^{\beta}&{\circ}\ar[r]^{\gamma}&{\circ}}$$ bounded by $\beta\gamma=0$,
and the tilting quiver $\overrightarrow{\mathcal{K}}_{B_0}$ is as follows:
$$\xymatrix{
\circ\ar[r]\ar@/^/@{.>}[rrrr]&\circ\ar[r]&\circ\ar[r]&\circ&\circ\ar[l]\\
                  &                  & \circ\ar[r]\ar[llu] &\circ.\ar[ur]   &&
}$$
For convenience, we will write it in this way:
$$\xymatrix{
{T^{'}_{12}}\ar[r]\ar@/^/@{.>}[rrrr]&{T^{'}_{13}}\ar[r]&{T^{'}_{9}}\ar[r]&{T^{'}_{10}}&{T^{'}_{11}}\ar[l]\\
                  &                  & {T^{'}_7}\ar[r]\ar[llu] &{T^{'}_6}\ar[ur]   &
}$$
such that $T^{'}_i=\Phi(T_i)$, for $i={6,7,9,10,11,12,13}$.
However, the shapes of $\overrightarrow{\mathcal{T}_{\mathcal{T}}}$ and $\overrightarrow{\mathcal{T}_{\mathcal{T},\mathcal{F}}}$ are as follows:

\begin{minipage}[t]{0.3\textwidth}
$$\xymatrix{T_9\ar@/^/@{.>}[rrrr]\ar[r]&T_{10}\ar[r]&T_{11}&&T_{13}\\
&&T_6\ar[u]
}$$
\end{minipage}
\begin{minipage}[t]{0.3\textwidth}
\begin{center}
\vspace{30pt}
and 
\end{center}
\end{minipage}
\begin{minipage}[t]{0.3\textwidth}
\begin{center}
\vspace{15pt}
$$\xymatrix{&  {T_7}\ar[r]  &  {T_{12}}\\           
}$$
\end{center}
\end{minipage}\\
So Theorem \ref{quiveriso} doesn't work any more and we cannot use the same way as in the first case to construct the tilting quiver of $B_0$
\end{example}


\Addresses

\end{document}